\documentclass[12pt,reqno]{amsart}
\usepackage{latexsym, amsfonts,mathrsfs, amsmath, amssymb, amscd, epsfig}
\usepackage{mathrsfs}

\usepackage{amsthm}
\usepackage[all,cmtip,line]{xy}
\usepackage{array}
\usepackage{psfrag}
\usepackage{overpic}
\usepackage{bm}
\usepackage{setspace}

\usepackage{color}

\newtheorem{theorem}{Theorem}
\newtheorem{lemma}{Lemma}[section]
\newtheorem{proposition}[lemma]{Proposition}
\newtheorem{definition}[lemma]{Definition}
\newtheorem{remark}[lemma]{Remark}

\newtheorem*{problemL}{{\bf{Main Problem}}}

\theoremstyle{definition}

\newenvironment{pf}{{\noindent \it  Proof of Theorem 1.}}{{\hfill$\Box$}\\}

\def\beq#1\eeq{\begin{equation}#1\end{equation}}
\def\balign #1 #2 \ealign{\begin{aligned} #1 #2  \end{aligned} }

\def\Div{{\rm div}}

\def\sgn{{\rm sgn}}

\def\bu{\mathbf{u}}
\def\bv{\mathbf{v}}

\newcommand \alp{\alpha}
\newcommand \eps{\varepsilon}

\newcommand \Gam{\Gamma}
\newcommand \gam{\gamma}
\newcommand \om{\omega}
\newcommand \tx{\text}
\newcommand \R{\mathbb{R}}
\newcommand \til{\tilde}

\newcommand \der{\partial}
\newcommand \mcl{\mathcal}
\newcommand \mfrak{\mathfrak}
\newcommand \ol{\overline}
\newcommand \Om{\Omega}

\newcommand \corners{\Gamw}

\newcommand \Sen{S_{en}}

\newcommand \pex{p_{ex}}

\newcommand \Gamen{\Gam_0}
\newcommand \Gamex{\Gam_L}
\newcommand \Gamw{\Gam_w}

\newcommand \rx{{\rm x}}
\newcommand \ry{{\rm y}}

\def\msB{\mathscr{B}}
\def\msK{\mathscr{K}}
\def\mff{\mathfrak{f}}
\def\mfh{\mathfrak{h}}
\def\mfF{\mathfrak{F}}
\def\mfg{\mathfrak{g}}

\newcommand \rhos{\rho_c}
\newcommand \mfraka{{a}}
\newcommand \mfrakb{{b}}
\newcommand \mfrakc{{c}}
\newcommand \mfrakd{{d}}
\newcommand \mfu{\mathfrak{u}}
\newcommand \mfV{\mcl{V}}

\newcommand \mfG{\mathfrak{G}}
\newcommand \mfH{\mathfrak{H}}
\newcommand \mfrakL{\mathfrak{L}}
\newcommand \mclL{\mathcal{L}}
\newcommand \mG{\mcl{G}}

\numberwithin{equation}{section}

\begin{document}

\title[Subsonic self-gravitation flows]
{Two dimensional subsonic flows with self-gravitation in bounded domain}


\author{Myoungjean Bae}
\address{M. Bae, Department of Mathematics\\
         POSTECH\\
         San 31, Hyojadong, Namgu, Pohang, Gyungbuk, Korea}
\email{mjbae@postech.ac.kr}
\author{Ben Duan}
\address{B. Duan, Department of Mathematics\\
         POSTECH\\
         San 31, Hyojadong, Namgu, Pohang, Gyungbuk, Korea}
\email{bduan@postech.ac.kr}

\author{Chunjing Xie}
\address{C. Xie, Department of mathematics, Institute of Natural Sciences, Ministry of Education Key Laboratory of Scientific and Engineering Computing, Shanghai Jiao Tong University\\
800 Dongchuan Road, Shanghai, China
}
\email{cjxie@sjtu.edu.cn}

\begin{abstract}
We investigate two dimensional steady Euler-Poisson system which describe the motion of compressible self-gravitating flows. The unique existence and stability of subsonic flows
in a duct of finite length are obtained when prescribing the entropy at the entrance and the pressure at the exit.
After introducing the stream function, the Euler-Poisson system can be decomposed into several transport equations and a second order nonlinear elliptic system.
We discover an energy estimate for the associated elliptic system which is a key ingredient to
prove  the unique existence and stability of subsonic flow. 

\end{abstract}


\date{\today}

\maketitle

\section{Introduction and Main Results}
\label{section-1}
The motion of self-gravitating flows can be described by the Euler-Poisson system
\begin{equation}
\label{1-a1}
\left\{
\begin{aligned}
&\rho_t+\Div(\rho{\bf u})=0,\\
&(\rho{\bf u})_t+\Div(\rho{\bf u}\otimes {\bf u}+p{\bf I})=-\rho\nabla\Phi,\\
&(\rho{\bf u} E)_t+\Div(\rho{\bf u}(E+\frac{p}{\rho}))=-\rho{\bf u}\cdot\nabla\Phi,\\
&\Delta \Phi=\rho,
\end{aligned}
\right.
\end{equation}
where $\Div, \nabla$, and $\Delta$ are the divergence, gradient, and Laplacian operators with respect to spatial variables $\rx\in\R^n$. The unkowns  $\rho, {\bf u}, p, E$ and $\Phi$ in \eqref{1-a1} represent the density, velocity, pressure, total energy of the fluid and gravitational potential function, respectively. ${\bf I}$ is an $n\times n$ identity matrix. For ideal polytropic gas, \eqref{1-a1} is closed with the aid of definition of total energy $E$ by
\begin{equation*}
E=\frac{|{\bf u}|^2}{2}+\frac{p}{(\gam-1)\rho},
\end{equation*}
where $\gam>1$ is the adiabatic constant.

The blowup of classical solutions and the existence of global weak radial solutions for the system \eqref{1-a1} were obtained in \cite{Makino, Wang}. The global well-posedness for \eqref{1-a1} in general multidimensional setting is an outstanding challenging problem. Two important classes of isentropic steady solutions of \eqref{1-a1} were studied extensively previously. The first is the non-rotating star solutions which have zero velocity, see \cite{Chandra}. The second is the rotating star solutions whose velocity fields are axially symmetric, \cite{CaF, DLYY, Friedman,Li, LS1, LS2,  McCann}, etc.
The stability and instability of non-rotating and rotating star solutions were investigated in \cite{Jang, Lin, LS3, Rein}, and references therein.

In this paper we focus on structural stability of  steady subsonic solutions to \eqref{1-a1}
when the boundary data are two dimensional small perturbations of  one dimensional solutions.
In $\R^2$, let $u$ and $v$ denote the horizontal and vertical components of the velocity ${\bf u}= (u, v)$. Then, the steady Euler-Poisson system is written as
\begin{equation}
\label{1-a3}
\left\{
\begin{aligned}
&(\rho u)_x+(\rho v)_y=0,\\
&(\rho u^2+p)_x+(\rho u v)_y=-\rho \Phi_x,\\
&(\rho u v)_x+(\rho v^2+p)_y=-\rho \Phi_y,\\
&(\rho u \msB)_x+(\rho v \msB)_y=-\rho {\bf u}\cdot \nabla\Phi,\\
&\Delta\Phi=\rho,
\end{aligned}
\right.
\end{equation}
where the Bernoulli function $\msB$ is given by
\begin{equation*}
\msB=E+\frac{p}{\rho}=\frac{|{\bf u}|^2}{2}+\frac{\gam p}{(\gam-1)\rho}.
\end{equation*}
Set
\begin{equation}
\label{def-sk}
S=\ln \frac{p}{\rho^{\gam}}\quad \tx{and}\quad \msK=\msB+\Phi.
\end{equation}
$S$ is the entropy, and we call $\msK$ {\emph{the pseudo-Bernoulli function}}.
$(\rho, u, v, p, \Phi)\in (C^1)^4\times C^2$ solve \eqref{1-a3} if and only if they solve
\begin{align}
\label{1-a5}
&(\rho u)_{x_1}+(\rho v)_{x_2}=0,\\
\label{1-a6}
&(\rho uv)_{x_1}+(\rho v^2+p)_{x_2}=-\rho\Phi_{x_2},\\
\label{1-a7}
&{\bf u}\cdot\nabla S=0,\\
\label{1-a8}
&{\bf u}\cdot\nabla \msK=0,\\
\label{1-a9}
&\Delta\Phi=\rho,
\end{align}
for ${\bf u}=(u,v)$
provided that $\rho>0$ and $u>0$.

As a nonlinear system for $(\rho, u, v, p)$, \eqref{1-a5}--\eqref{1-a8} form a mixed type system, and its type depends on \emph{the Mach number} $M$ which is given
by $M=\frac{|\bu|}{c}$ with $c(\rho,p)=\sqrt{\frac{\gam p}{\rho}}$. Here, $c$ is called the \emph{local sound speed}. If $M<1$, the flow is said to be \emph{subsonic},
then \eqref{1-a5}--\eqref{1-a8} form an elliptic-hyperbolic coupled system. If $M>1$, the flow is said to be \emph{supsersonic} and \eqref{1-a5}--\eqref{1-a8} form a hyperbolic system.
In addition,
the Poisson's equation \eqref{1-a9} has a nonlocal effect to the other equations \eqref{1-a5}-\eqref{1-a8}, and it makes the fluid variables $\rho, {\bf u}, p$ and gravitational potential $\Phi$ interact in a highly nonlinear way.

The goal of this paper is to prove unique existence and stability of subsonic flows for
\eqref{1-a5}-\eqref{1-a9} in a duct when the entrance entropy and the exit pressure are prescribed by two dimensional small perturbations of one dimensional solutions.
For that purpose, we first study the one dimensional solutions of \eqref{1-a5}-\eqref{1-a9}.


\subsection{One dimensional solutions of \eqref{1-a5}-\eqref{1-a9}} Consider a solution $(\rho, u, v, p, \Phi)$ of  \eqref{1-a5}-\eqref{1-a9} with $v=0=\Phi_{x_2}$,
$\rho>0$ and $u>0$. Set $G:=\Phi_{x_1}$. Then $(\rho, u, S, G)$ satisfy
\begin{align}
\label{1-b1}
&(\rho u)'=0,\\
\label{1-b2}
&S'=0,\\
\label{1-b3}
&\msB'=-G,\\
\label{1-b4}
&G'=\rho,
\end{align}
where $'$ denotes the derivative with respect to $x_1$. It follows
from \eqref{1-b1} and \eqref{1-b2} that one has
\begin{equation}
\label{1-b5}
u=\frac{m_0}{\rho}\quad\text{and}\quad  S=S_0
\end{equation}
 with the constants $m_0>0$ and $S_0>0$ determined by the data at the entrance. Then \eqref{1-b3} and \eqref{1-b4} can be written as 
\begin{equation}
\label{1-b6}
\begin{cases}
\rho'=\frac{-\rho G}{\gam e^{S_0}\rho^{\gam-1}-\frac{m_0^2}{\rho^2}},\\
G'=\rho.
\end{cases}
\end{equation}
For fixed constants $\gam>1$, $m_0>0$, and $S_0>0$, set
\begin{equation*}
\rhos=\left(\frac{m_0^2}{\gam e^{S_0}} \right)^{\frac{1}{\gam+1}}.
\end{equation*}
\begin{proposition}
\label{proposition-1}
Fix constants $\gam>1$, $m_0>0$ and $S_0>0$.
\begin{itemize}
\item[(a)] For any given $\rho_0\in \R_+\setminus\{0,\rhos\}$ and $G_0\in\R$, there exists an $\bar{L}>0$ depending on $\gam$, $m_0$, $S_0$, $\rho_0$, $G_0$ such that
the ODE system \eqref{1-b6} with initial conditions
    \begin{equation}
    \label{1-b8}
\rho(0)=\rho_0\quad \text{and}\quad G(0)=G_0
    \end{equation}
    has a unique smooth solution $(\rho, G)$ on the interval $[0,L]$ whenever $L\le\bar L$, and
   \begin{equation*}
   \lim_{x_1\to \bar{L}-}(\rho, G)=(\rho_c, G_M)
   \end{equation*}
with
\begin{equation}
\label{1-c4}
G_M=\sqrt{2\left(e^{S_0}\rho_0^\gamma +\frac{m_0^2}{\rho_0}-\frac{\gamma+1}{\gamma}\left(\frac{m_0}{\gamma e^{S_0}}\right)^{-\frac{1}{\gamma+1}}m_0^2\right)+G_0^2}.
\end{equation}

\item[(b)] Let $\bar L(\rho_0, G_0)$ be the lifespan of the initial value problem \eqref{1-b6} and \eqref{1-b8}. If $\gam>2$, then
\begin{equation*}
\lim_{{\rho_0}\to \infty}\bar{L}(\rho_0, G_0)=\infty.
\end{equation*}
In other words, if $\gam>2$, for any $L>0$, there exists a nonempty set $\mfrak{P}_1\subset (\rhos,\infty)\times \R$ depending on $\gam$, $m_0$, $S_0$, $L$ so that whenever $(\rho_0,G_0)\in \mfrak{P}_1$, the problem \eqref{1-b6} and \eqref{1-b8} has a unique smooth solution $(\rho,G)$ on the interval $[0,L]$.
\end{itemize}
\end{proposition}

\begin{proof}
The right hand sides of \eqref{1-b6} are smooth functions of $(\rho, G)$ unless $\rho=\rho_c$. So the unique existence theorem of ODEs implies that if $\rho_0\neq \rho_c$, then \eqref{1-b6} with \eqref{1-b8} is uniquely solvable on $[0,L]$ for a small $L>0$.

If $(\rho, G)\in (C^1([0,L]))^2$ solve \eqref{1-b6} and \eqref{1-b8}, then
\begin{equation}
\label{1-c1}
H(\rho)+\frac 12 G^2=H(\rho_0)+\frac 12 G_0^2  
\end{equation}
holds on $[0,L]$ for $H(\rho)$ defined by
\begin{equation}
\label{1-c2}
H(\rho)=\int_{\rho_c}^{\rho}\gam e^{S_0}\varrho^{\gam-1}-\frac{m_0^2}{\varrho^2} \;d\varrho
=e^{S_0}\rho^{\gam} +\frac{m_0^2}{\rho}-\frac{\gam+1}{\gamma}\left(\frac{m_0^2}{\gam e^{S_0}} \right)^{-\frac{1}{\gam+1}} m_0^2.
\end{equation}
 It is easy to check that
\begin{equation}
\label{1-c3}
H(\rho)>0\quad \text{and}\quad
\sgn H'(\rho)=\sgn(\rho-\rhos)\;\;\tx{for}\;\;\rho>0.
\end{equation}
Thus we can draw the phase plane as Figure \ref{figure-1}.
\begin{figure}[htp]
\centering
\begin{psfrags}
\psfrag{g}[cc][][0.8][0]{$G$}
\psfrag{r}[cc][][0.8][0]{$\rho$}
\psfrag{rs}[cc][][0.8][0]{$\rho=\rhos$}
\psfrag{i}[cc][][0.8][0]{$\phantom{aaa}(\rho_0,G_0)$}
\psfrag{I}[cc][][0.8][0]{$\phantom{aaaaa}(\rho_0,G_0)\phantom{a}$}
 \hfill\includegraphics[scale=.8]{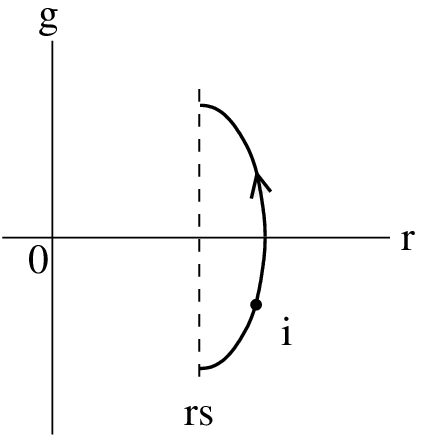}\hfill
\includegraphics[scale=.8]{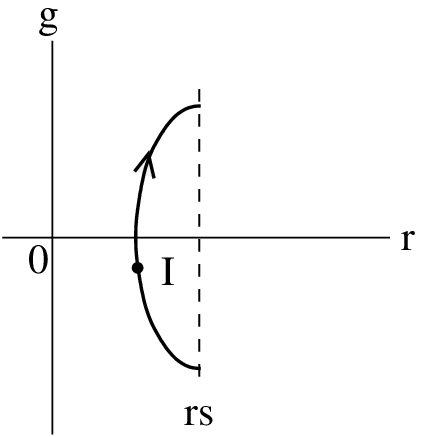}\hfill\phantom{a}
\caption{Left: $\rho_0>\rhos$,\quad Right: $0<\rho_0<\rhos$}
\label{figure-1}
\end{psfrags}
\end{figure}

It follows from \eqref{1-c3} that there exists a unique $\rho_M$ satisfying
\begin{equation*}
H(\rho_M)=H(\rho_0)+\frac{1}{2} G_0^2\quad \text{and}\quad (\rho_M-\rhos)(\rho_0-\rhos)>0.
\end{equation*}
It is obvious that $(\rho_M,0)$ and $(\rhos, G_M)$ with $G_M$ defined in \eqref{1-c4} lie on the curve of \eqref{1-c1}. Suppose that $\rho_0>\rhos$. Then
the initial value problem
\eqref{1-b6} and \eqref{1-b8} has a unique smooth solution $(\rho, G)$ as long as $G<G_M$. Fix a small constant $\eps>0$ such that $\eps<\frac 12(G_M-G_0)$. Then  the initial value problem \eqref{1-b6} and \eqref{1-b8} has a unique smooth solution on the interval $[0,L_{\eps}]$ where $G(L_{\eps})=G_M-\eps$ and $\rho(L_{\eps})>\rhos$. Since $\rho>\rho_c$, it follows from the equation \eqref{1-b4} that
\[
L_{\eps}<\frac{G_M-G_0}{\rho_c}.
\]
Note that $L_{\eps}$ continuously increases as $\eps>0$ decreases to $0$. Therefore, there exists a constant $\bar{L}\in(0, \frac{G_M-G_0}{\rho_c} ]$ depending on $\gam$, $m_0$, $S_0$, $\rho_0$, $G_0$ such that
\begin{equation*}
\bar{L}=\sup_{\eps>0}L_{\eps}.
\end{equation*}
Since $G$ tends to $G_M$ as $x_1$ converges to $\bar{L}-$, the identity \eqref{1-c1} implies that $\rho$ tends to $\rho_c$ at the same time. This proves (a) for the case $\rho_0>\rho_c$. The case $\rho_0<\rho_c$ can be proved
similarly to obtain (a), so we omit details.

To prove (b), we need to estimate a lower bound of $\bar{L}$. Fix $\gam>2$ and $\rho_0>\rho_c$. For a small $\eps>0$, it follows from
$G'=\rho$ that
\begin{equation}
\label{1-c6}
L_{\eps}\ge \frac{1}{\rho_M}\int_0^{L_{\eps}}\rho(x_1)\;dx_1=\frac{G_M-G_0-\eps}{\rho_M}
\end{equation}
Letting $\eps>0$ go to $0$ in \eqref{1-c6} yields
\begin{equation*}
\bar{L}(\rho_0)\ge \frac{G_M-G_0}{\rho_M}\ge \frac{G_M}{\rho_M}-\frac{G_0}{\rho_0}
\end{equation*}
where $\rho_M>\rho_0$ is used. It follows from \eqref{1-c1} and \eqref{1-c2} that $\frac{G_M}{\rho_M}$ can be represented as
\begin{equation*}
\frac{G_M}{\rho_M}=\sqrt{\frac{2H(\rho_M)}{\rho_M^2}}.
\end{equation*}
Fix $G_0$ and let $\rho_0$ tend to infinity. Then, $\rho_M$ tends to infinity as well.
The explicit form of $H(\rho)$ in \eqref{1-c2}, together with $\gam>2$, gives
\[
\lim_{\rho_M\to \infty}\frac{G_M}{\rho_M}=\infty.
\]
This proves (b).



\end{proof}

\begin{lemma}
\label{lemma-2}
Given $\rho_0\in \R_+\setminus\{0, \rho_c\}$ and $G_0\in \R$, let $\bar{L}$ be from (a) of Proposition \ref{proposition-1}. Then the solution $(\rho, G)$ of \eqref{1-b6} and \eqref{1-b8} satisfies
\begin{equation*}
{\rm sgn}(\rho-\rho_c)={\rm sgn}(\rho_0-\rho_c)\quad\tx{on}\quad [0, \bar{L}).
\end{equation*}
Therefore, for any $L\in(0, \bar{L})$, there exists a constant $\nu_0>0$ depending on $\gam$, $m_0$, $S_0$, $\rho_0$, $G_0$ and $L$ such that
\begin{equation}
    \label{1-b9}
    (\gam e^{S_0}\rho^{\gam-1}-\frac{m_0^2}{\rho^2}){\rm sgn}(\rho_0-\rho_c)\ge \nu_0 \;\;\tx{on}\;\;[0,L].
    \end{equation}

\end{lemma}
\begin{proof}
The lemma is a direct consequence of the equation \eqref{1-b4}, the identity \eqref{1-c1}, and the property \eqref{1-c3}.
\end{proof}



\subsection{
Main theorems}
Given $\gam>1$, $S_0>0$ and $m_0>0$,
let $(\rho, G)$ solve \eqref{1-b6} and \eqref{1-b8} with $\rho_0>\rhos$ on $[0, L]$ for some $L<\bar{L}$.
Define $\Om_L:=(0,L)\times (-1,1)$.
The boundary $\der\Om_L$ consists of
\begin{equation*}
\begin{split}
&\Gamen=\der\Om_L\cap\{x_1=0\},\quad \Gamex=\der\Om_L\cap\{x_1=L\},\quad
\Gamw=\der\Om_L\setminus(\Gamen\cup\Gamex).
\end{split}
\end{equation*}
For $\rx=(x_1, x_2)\in \Omega_L$,
set
\begin{equation}
\label{backgroundsol}
(\bar\rho(\rx), \bar{\bf u}(\rx), \bar p(\rx)):=(\rho(x_1), u(x_1), 0, p(x_1))
\end{equation}
where $u$ is given in \eqref{1-b5} and $p=e^{S_0}\rho^\gamma$. Define
\begin{equation}
\label{1-d3}
\Phi_0(\rx):=\int_0^{x_1} G(t)dt-\left(\frac{m_0^2}{2\rho_0^2}+\frac{\gam e^{S_0}\rho_0^{\gam-1}}{\gam-1} \right).
\end{equation}
Then $(\bar\rho, \bar{\bf u}, \bar p, \Phi_0)(\rx)$ satisfy \eqref{1-a5}--\eqref{1-a9} in $\Om_L$. Also, $\bar{\bf u}$ and $\Phi_0$ satisfy
\begin{equation*}
\bar{\bf u}\cdot{\bf n}_w=0\quad\text{and}\quad \nabla\Phi_0\cdot{\bf n}_w=0\quad\tx{on}\quad \Gamw
\end{equation*}
where ${\bf n}_w$ is the inward unit normal  of $\Gamw$. Furthermore,
by the definition \eqref{1-d3} of $\Phi_0$, we have
\begin{equation*}
\bar{\msB}+\Phi_0\equiv 0
\quad{\tx{in}}\quad \Om_L,
\end{equation*}
for
\begin{equation*}
\bar{\msB}=\frac 12|\bar{\bf u}|^2+\frac{\gam \bar p}{(\gam-1)\bar{\rho}},\quad
\msB_0=\bar{\msB}(0).
\end{equation*}
\begin{definition}
\label{def-background}
For fixed $\gam>1$, $(\bar{\rho}, {\bar{\bf u}}, \bar p, \Phi_0)$ given by \eqref{backgroundsol} is called the background solution corresponding to $(S_0, m_0, \rho_0, G_0)$.
\end{definition}
Our aim is to prove stability of the background solution $(\bar{\rho}, \bar{\bf u}, \bar p, \Phi_0)$ under two dimensional small perturbations of boundary data.
Before we state our main problem and main theorems, weighted H\"{o}lder norms are introduced first.
For a bounded connected open set $\mcl{O}\subset\R^2$, let $\Gam$ be a closed portion of $\der\mcl{O}$. For $\rx,\ry\in\mcl{O}$, set
\begin{equation*}
\delta_{\rx}:=\rm{dist}(\rx,\Gam)\quad \text{and}\quad  \delta_{\rx,\ry}:=\min(\delta_{\rx},\delta_{\ry}).
\end{equation*}
For $k\in\R$, $\alp\in(0,1)$ and $m\in \mathbb{Z}^+$, define the standard H\"{o}lder norms by
\begin{align*}
&\|u\|_{m,\mcl{O}}:=\sum_{0\le|\beta|\le m}\sup_{\rx\in \mcl{O}}|D^{\beta}u(\rx)|,\quad
[u]_{m,\alp,\mcl{O}}:=\sum_{|\beta|=m}\sup_{\rx, \ry\in\mcl{O},\rx\neq  \ry}\frac{|D^{\beta}u(\rx)-D^{\beta}u(\ry)|}{|\rx-\ry|^{\alp}},
\end{align*}
and the weighted H\"{o}lder norms by
\begin{align*}
&\|u\|_{m,0,\mcl{O}}^{(k,\Gam)}:=\sum_{0\le|\beta|\le m}\sup_{\rx\in \mcl{O}}\delta_{\rx}^{\max(|\beta|+k,0)}|D^{\beta}u(\rx)|,\\
&[u]_{m,\alp,\mcl{O}}^{(k,\Gam)}:=\sum_{|\beta|=m}\sup_{\rx,\ry\in\mcl{O}, \rx\neq \ry}\delta_{\rx,\ry}^{\max(m+\alp+k,0)}\frac{|D^{\beta}u(\rx)-D^{\beta}u(\ry)|}{|\rx-\ry|^{\alp}},\\
&\|u\|_{m,\alp,\mcl{O}}:=\|u\|_{m,\mcl{O}}+[u]_{m,\alp,\mcl{O}},\quad
\|u\|_{m,\alp,\mcl{O}}^{(k,\Gam)}:=\|u\|_{m,0,\mcl{O}}^{(k,\Gam)}+[u]_{m,\alp,\mcl{O}}^{(k,\Gam)},
\end{align*}
where $D^{\beta}$ denotes $\der_{x_1}^{\beta_1}\der_{x_2}^{\beta_2}$ for a multi-index $\beta=(\beta_1,\beta_2)$ with $\beta_j\in\mathbb{Z}_+$ and $|\beta|=\beta_1+\beta_2$. $C^{m,\alp}_{(k,\Gam)}(\mcl{O})$ denotes the completion of the set of all smooth functions whose $\|\cdot\|_{m,\alp,\mcl{O}}^{(k,\Gam)}$ norms are finite. For simplicity of notations, let $\|\cdot\|_{\alp,\mcl{O}}$ denote $\|\cdot\|_{0,\alp,\mcl{O}}$. For a vector function $\bv=(v_1, \cdots, v_n)$, define $\|\bv\|_{m, \alpha, \mcl{O}}^{k, \Gam}= \sum_{i=1}^n\|v_i\|_{m, \alpha, \mcl{O}}^{k, \Gam}$ and $\|\bv\|_{W^{k, p}(\mcl{O})}= \sum_{i=1}^n\|v_i\|_{W^{k, p}(\mcl{O})}$.
\begin{problemL}
Fix a constant $\alp\in(0,1)$.
For given functions $(\Phi_{bd}, p_{ex}, \msB_{en}, S_{en})$ satisfying
\begin{equation}
\label{1-d7}
\begin{split}
&\|G_{en}-G_0\|_{1,\alp,\Gamen}^{(-\alp,\der\Gamen)}+
\|\msB_{en}-\msB_0\|_{1,\alp,\Gamen}^{(-\alp, \partial\Gamen)}+\|S_{en}-S_0\|_{1,\alp,\Gamen}\\
&\phantom{ttttttttttttaaaaaaa}+\|\Phi_{bd}-\Phi_0\|_{2,\alp,\Gamex}^{(-1-\alp,\Gamw)}
+\|p_{ex}-\bar p(L)\|_{1,\alp,\Gamex}^{(-\alp,\partial\Gamex)}
\le \sigma
\end{split}
\end{equation}
for a sufficiently small constant $\sigma>0$, find a solution $(\rho,u,v, p, \Phi)$ of \eqref{1-a3}
in $\Om_L$ with boundary conditions
\begin{align}
\label{1-d8}
&v=0,\,\,
\Phi_{x_1}=G_{en},\,\,
S=S_{en},\,\,
\msB=\msB_{en}
\,\,&\tx{on}\quad \Gamen,\\
\label{1-d9}
& (u,v)\cdot{\bf n}_w=\nabla\Phi\cdot{\bf n}_w=0&\tx{on}\quad \Gamw,\\
\label{1-d10}
&p=p_{ex},\,\,
\Phi=\Phi_{bd}
&\tx{on}\quad \Gamex,
\end{align}
where ${\bf n}_w$ is the inward unit normal vector on $\Gamw$.
\end{problemL}

Before stating our main results, we introduce the following notations.
Fix constant $\alp\in(0,1)$, $\mu>1$.  For functions $(\Phi_{bd},G_{en}, \Sen, \msB_{en},\pex)\in C^2(\Om_L)\times [C^1(\Gamen)]^3
\times C^1(\Gamex)$, set
\begin{equation*}
\begin{split}
&\om_1(G_{en}, \Phi_{bd},S_{en},\msB_{en},p_{ex}):=\|G_{en}-G_0\|_{1,\alp,\Gamen}^{(-\alp, \der\Gamen)}
+\|S_{en}-S_0\|_{1,\alp,\Gamen}+\|\msB_{en}-\msB_0\|_{1,\alp,\Gamen}^{(-\alp,\partial \Gamen)}
\\
&\phantom{aaaaaaaaaaaaaaaaaaaaaaa}+\|\Phi_{bd}-\Phi_0\|_{2,\alp,\Gamex}^{(-1-\alp,\der\Gamex)}
+\|p_{ex}-\bar p(L)\|_{1,\alp,\Gamex}^{(-\alp,\partial\Gamex)},\\
&\om_2(\Phi_{bd},
\Sen, \mathscr{B}_{en}):=\|\Sen-S_0\|_{W^{2,\mu}(\Gamen)}
+\|\mathscr{B}_{en}+\Phi_{bd}\|_{W^{2,\mu}(\Gamex)}.
\end{split}
\end{equation*}

The main results of this paper is as follows:
\begin{theorem}
\label{theorem-3}
Fix $\gam>1$, and let $(\bar{\rho}, {\bar{\bf u}}, \bar{p}, \Phi_0)$ be the background solution corresponding to $(S_0, m_0, \rho_0, G_0)$ in $\Om_L$ with $\rho_0>\rhos$. Suppose that the background solution satisfies
\begin{equation}
\label{estd0}
\frac{\bar{\rho}L^2}{2(\gam e^{S_0}{\bar{\rho}}^{\gam-1}-\frac{m_0^2}{{\bar{\rho}}^2})}\le 1-\delta_0\quad\tx{in}\quad \Om_L
\end{equation}
for some constant $\delta_0\in(0,1)$.

(a)[Existence] Suppose that
$\Phi_{bd}$ satisfies the compatibility condition
\begin{equation}
\label{1-f10}
\der_{x_2}\Phi_{bd}=0\quad\tx{on}\quad \ol{\Gamw}\cap \ol{\Gamex}.
\end{equation}
Then there exists a small constant $\sigma_1$ depending on $\gam$, $m_0$, $S_0$, $\rho_0$, $G_0$, $L$, $\delta_0$ and $\alp$ such that
if
\begin{equation}
\label{pert1}
\om_1(G_{en},\Phi_{bd}, S_{en}, \msB_{en}, \pex)\le \sigma_1
\end{equation}
for some $\alp\in(0,1)$, then
the nonlinear boundary value problem \eqref{1-a3} with boundary conditions \eqref{1-d8}--\eqref{1-d10} has a solution $(\rho, u, v, p,\Phi)\in [C(\ol{\Om_L})\cap C^1(\Om_L)]^4\times (C^1(\ol{\Om_L})\cap C^2(\Om_L))$ satisfying
the estimate
\begin{equation}
\label{2-d3}
\|(\rho, u, v, p)-(\bar{\rho}, \bar u, 0, \bar p)\|_{1,\alp,\Om_L}^{(-\alp,\corners)}
+\|\Phi-\Phi_0\|_{2,\alp,\Om_L}^{(-1-\alp,\corners)}\le C\om_1(G_{en}, \Phi_{bd},\Sen,\msB_{en},\pex),
\end{equation}
where the constant $C$ is determined by $\gam$, $m_0$, $\rho_0$, $G_0$, $S_0$, $L$, $\delta_0$ and $\alp$.

(b) [Uniqueness] There exists $\sigma_2>0$ depending on $\gam$, $m_0$, $\rho_0$, $G_0$, $S_0$, $L$, $\delta_0$, $\alp$ and $\mu$ such that if
\begin{equation*}
\om_1(G_{en}, \Phi_{bd},S_{en}, \msB_{en},p_{ex})+
\om_2(\Phi_{bd},
\Sen, \mathscr{B}_{en})\le \sigma_2,
\end{equation*}
with $\alp\in(\frac 12, 1)$ and $\mu\in(2,\infty)$,
then the solution $(\rho, {\bf u}, p, \Phi)$ obtained in (a) is unique.
\end{theorem}

\begin{remark} \label{remark-condition}
There are a large class of background solutions for which the condition
 \eqref{estd0} holds. For example, we can obtain background solutions satisfying \eqref{estd0} in the following two cases.
\begin{itemize}
\item[(i)] For any given $m_0>0, \gam>1, G_0$ and $\rho_0(>\rhos)$,  if $L$ is suitably small, then the condition \eqref{estd0} holds.
\item[(ii)] Suppose that $\gam>2$. Let $(\bar{\rho}, \bar{G})$ be a background solution with $\rho_0>\rhos$. If $\bar{G}\le 0$ in $\Om_L$, then
\begin{equation*}
\max_{\ol{\Om_L}} \frac{\bar{\rho}}{\gam e^{S_0}{\bar{\rho}}^{\gam-1}-\frac{m_0^2}{{\bar{\rho}}^2}}\le
\frac{\rho_0}{\gam e^{S_0}\rho_0^{\gam-1}-\frac{m_0^2}{\rho_0^2}}\rightarrow 0\quad\tx{as}\quad \rho_0\to \infty.
\end{equation*}
Therefore, one can choose $\rho_0>\rhos$ large depending on $\gam$, $m_0$, $S_0$, $G_0$, $L$ so that \eqref{estd0} holds.
\end{itemize}
\end{remark}

\begin{remark}
\label{remark-condition2}
Theorem \ref{theorem-3} is proved by  a fixed point method, and \eqref{estd0} is needed to guarantee the well-posedness of a linear elliptic system boundary value problem \eqref{3-a1}--\eqref{3-a1-b2} related to \eqref{1-a5}--\eqref{1-a9}.  If the condition \eqref{estd0} is violated, in general, the associated boundary value problem with linear elliptic system may not be wellposed.
\end{remark}

The main idea to prove Theorem \ref{theorem-3} is to introduce a stream function $\psi$  which
reduces \eqref{1-a3} to a nonlinear elliptic system for $(\psi, \Psi)$ and two
transport equations for $(S, \msK)$.
The second order elliptic system for $(\psi,\Psi)$ is solved with the aid of its elaborate structure.
The technique of this work was inspired by the study on two dimensional subsonic Euler equations \cite{XX3}
and subsonic Euler-Poisson system modeling the flows in semiconductor devices \cite{BDX}.


In the rest of the paper, we say that a constant $C$ depends on the background data if $C$ is chosen depending on
$\gamma$, $m_0$, $S_0$, $\rho_0$, $G_0$, and $L$.
Hereafter, any constant $C>0$ appearing in various estimates is presumed to depend on the
background data
unless otherwise specified.

The rest of the paper is organized as follows. Using a stream function $\psi$, \eqref{1-a3} is
reduced to a nonlinear second order elliptic system for $(\psi,\Phi)$ and transport equations for $S$ and $\msK$ in Section \ref{section-2}. Various a priori estimates and the existence of solutions to the linear boundary value problems are given in Section \ref{section-3}. We give the proof for Theorem \ref{theorem-1}, an equivalence of Theorem \ref{theorem-3}, in Section \ref{section-4}.

\section{Stream function formulation for the problem}
\label{section-2}
\subsection{Stream function formulation and Proof of Theorem \ref{theorem-3}}
We will prove Theorem \ref{theorem-3} for $\sigma_1$ sufficiently small so that any solution $(\rho, u, v, p, \Phi)$ satisfying \eqref{2-d3} satisfies $\rho>0, u>0$. In that case, \eqref{1-a3} is equivalent to \eqref{1-a5}--\eqref{1-a9}.

Suppose that $(\rho, u, v, p, \Phi)\in [C^1(\Omega)]^4\times C^2(\Omega)$ is a solution to \eqref{1-a5}--\eqref{1-a9}. By \eqref{1-a5}, there is a $C^2$ function $\psi$ satisfying
\begin{equation}
\label{1-e1}
\rho(u,v)=\nabla^{\perp}\psi\quad\tx{where}\quad \nabla^{\perp}:=({\der_{x_2}}, -{\der_{x_1}}).
\end{equation}
Note that there is a freedom of choice for the value of $\psi(0,-1)$.
Without loss of generality, we assume $\psi(0, -1)=0$.
It follows from \eqref{def-sk} and \eqref{1-e1} that
\begin{equation}
\label{1-e0}
\mG(\rho,\nabla\psi,\Phi,S,\msK)=0\quad\tx{in}\quad\Om_L,
\end{equation}
where for ${\bm q}=(q_1, q_2)$, $\mG$ is defined by
\begin{equation}
\label{1-e8}
\mG(\varsigma,{\bm q}, z, s, \eta)=\frac{|{\bm q}|^2}{\varsigma^2}+\frac{\gam \varsigma^{\gam-1} \exp(s)}{\gam-1}+z-\eta.
\end{equation}
If $\rho>0$ and $u>0$, \eqref{1-a6} is equivalent to
\begin{equation*}
v_{x_1}-u_{x_2}=\left(\frac{e^S\rho^{\gam-1}S_{x_2}}{\gam-1} -\msK_{x_2}\right)\frac{1}{u}.
\end{equation*}
This, together with \eqref{1-e1}, gives
\begin{equation}
\label{1-e5}
\Div\left(\frac{\nabla\psi}{\rho}\right)=
\left(\msK_{x_2}-\frac{Q^{\gam-1}\exp(S) S_{x_2}}{\gam-1}\right)\frac{\rho}{\psi_{x_2}}.
\end{equation}
It follows from \eqref{1-e1} that if $\rho>0$, then the equations \eqref{1-a7} and \eqref{1-a8} can be written as
\begin{align}
\label{1-f4t}
\nabla^{\perp}\psi\cdot\nabla S=0\quad \text{and}\quad
\nabla^{\perp}\psi\cdot\nabla \msK=0,
\end{align}
respectively.
Finally, the boundary conditions \eqref{1-d8}--\eqref{1-d10} can be formulated as the boundary conditions for $(\rho, \psi, \Phi, S, \msK)$ as follows:
\begin{equation}
\label{1-f5}
\left\{
\begin{aligned}
&\psi_x=0,\,\,
\Phi_{x_1}=G_{en},\,\,
S=S_{en},\,\,
\msK=\msB_{en}+\Phi_{bd} \,\,\tx{on}\;\;\Gamen,\\
&\psi_{x_1}=0,\,\,\Phi_{x_2}=0 \,\,\tx{on}\;\;\Gamw,\\
&e^S\rho^{\gam}=p_{ex},\,\, \Phi=\Phi_{bd} \,\,\tx{on}\;\;\Gamex.
\end{aligned}
\right.
\end{equation}
Under the condition of $\rho>0, u>0$, $(\rho, u, v, p, \Phi)$ solve \eqref{1-a3}  in $\Om_L$ with \eqref{1-d8}--\eqref{1-d10} if and only if $(\rho, \psi, \Phi, S, \msK)$ solve
\eqref{1-a9},  \eqref{1-e5}, \eqref{1-f4t} in $\Om_L$
with \eqref{1-f5}.

Define
\begin{equation}
\label{1-e6}
\psi_0(x_1,x_2)=m_0(x_2+1)\quad\tx{in}\quad\Om_L.
\end{equation}
Then, $(\psi, \Phi, \rho, S,\msK)=(\psi_0, \Phi_0, \bar{\rho}, S_0, 0)$ satisfy the system \eqref{1-a9},  \eqref{1-e5}, and \eqref{1-f4t} with boundary conditions \eqref{1-f5} where $(G_{en}, \Phi_{bd}, \Sen, \msB_{en}, \pex)=(G_0, \Phi_0, S_0, \msB_0, \bar{p}(L))$.

\begin{theorem}
\label{theorem-1}
Fix $\gam>1$, and let $(\bar{\rho}, {\bar{\bf u}}, \bar{p}, \Phi_0)$ be the background solution corresponding to $(S_0, m_0, \rho_0, G_0)$ in $\Om_L$ with $\rho_0>\rhos$. Suppose that the background solution satisfies \eqref{estd0}.
\begin{enumerate}
\item[(a)] [Existence]
Suppose that $\Phi_{bd}$ satisfies \eqref{1-f10}.
Then there exist $\sigma_3>0$ small and $C>0$ depending on the background data, $\delta_0$ and $\alp$ such that if $\om_1(G_{en}, \Phi_{bd},S_{en},\msB_{en},p_{ex})\le \sigma_3$ holds, then
the nonlinear boundary value problem  \eqref{1-a9},  \eqref{1-e5}, \eqref{1-f4t}, and \eqref{1-f5} has a solution $(\psi,\Phi, \rho, S,\msK)\in [C^2(\Om_L)]^2\times [C^1(\Om_L)]^3$, which satisfies
\begin{equation}
\label{1-f12}
\begin{split}
&\|(\psi-\psi_0, \Phi-\Phi_0)\|_{2,\alp,\Om_L}^{(-1-\alp,\corners)}
+\|(\rho-\bar{\rho},S-S_0, \msK)\|_{1,\alp,\Om_L}^{(-\alp,\Gamw)}\\
\le & C\om_1(G_{en}, \Phi_{bd},S_{en}, \msB_{en},p_{ex}).
\end{split}
\end{equation}
\item[(b)][Uniqueness]
There exists a constant $\sigma_4>0$ depending on the background data, $\delta_0$, $\alp$ and $\mu$ such that if
\begin{equation}
\label{1-f13}
\om_1(G_{en}, \Phi_{bd},S_{en}, \msB_{en},p_{ex})+
\om_2(\Phi_{bd},
\Sen, \mathscr{B}_{en})\le \sigma_4,
\end{equation}
with $\alp\in (\frac 12, 1)$ and $\mu\in(2,\infty)$,
then the solution $(\rho, \psi,\Phi, S,\msK)$
obtained in (a) is unique.
\end{enumerate}
\end{theorem}
Once Theorem \ref{theorem-1} is proved, then Theorem \ref{theorem-3} easily follows from Theorem \ref{theorem-1}.

\begin{pf}
Suppose that Theorem \ref{theorem-1} is true, then fix $(\sigma_1,\sigma_2)=(\sigma_3,\sigma_4)$ for $(\sigma_3, \sigma_4)$ from Theorem \ref{theorem-1}. Given $(G_{en}, \Phi_{bd}, \Sen, \msB_{en}, \pex)$ satisfying \eqref{1-f10} and \eqref{pert1}, let $(\psi, \Phi, \rho, S, \msK)$ be a solution to \eqref{1-a9},  \eqref{1-e5}, \eqref{1-f4t}, and \eqref{1-f5}.  Since $\sigma_3$ is suitably small, the estimate \eqref{1-f12} implies that $\rho>0$.  Set
$
{\bf u}=\frac{\nabla^{\perp}\psi}{\rho}$ and $p=e^S\rho^{\gam}$. Then, $(\rho, {\bf u}, p, \Phi)$ solve \eqref{1-a3} with \eqref{1-d8}--\eqref{1-d10}, and satisfy \eqref{2-d3}.

 Given $(G_{en}, \Phi_{bd}, \Sen, \msB_{en}, \pex)$ satisfying \eqref{1-f10} and \eqref{1-f13}, let $(\rho_1, u_1, v_1, p_1, \Phi_1)$ and
 $(\rho_2, u_2, v_2, p_2, \Phi_2)$ be two solutions of \eqref{1-a3}, \eqref{1-d8}--\eqref{1-d10} satisfying \eqref{2-d3}. For each $k=1,2$, define
 \begin{equation*}
 \begin{split}
 &\psi_k(x_1,x_2)=\int_{-1}^{x_2}\rho_ku_k(x_1, s)\;ds,\\
 &S_k=\ln\frac{p_k}{\rho_k^{\gam}},\quad \msK_k=\frac 12(u_k^2+v_k^2)+\frac{\gam p_k}{(\gam-1)\rho_k}+\Phi_k.
 \end{split}
 \end{equation*}
Then, each $(\psi_k, \Phi_k, \rho_k, S_k, \msK_k)$ is a solution to \eqref{1-a9},  \eqref{1-e5}, \eqref{1-f4t}, \eqref{1-f5}, and satisfies \eqref{1-f12}. Then, Theorem \ref{theorem-1}(b) implies $(\psi_1, \Phi_1, \rho_1, S_1, \msK_1)=(\psi_2, \Phi_2, \rho_2, S_2, \msK_2)$ in $\Om_L$ from which Theorem \ref{theorem-3}(b) follows.
\end{pf}

The rest of  the paper is devoted to prove Theorem \ref{theorem-1}.

\subsection{Problem for the perturbations}
\label{subsec-pert}
Let $(\bar{\rho}, \bar{\bf u}, \bar{p}, \Phi_0)$ be a background solution satisfying all the assumptions in Theorem \ref{theorem-1}. By \eqref{1-e8} and Lemma \ref{lemma-2},
there is a constant $\kappa_0>0$ depending on the background data to satisfy
\begin{equation}
\label{1-e9}
\mG_{\varsigma}(\bar{\rho},\nabla\psi_0,\Phi_0,S_0,0)=
\frac{1}{\bar{\rho}}\left(\gam \bar{\rho}^{\gam-1} \exp(S_0)-\frac{m_0^2}{\bar{\rho}^2}\right)\ge
\kappa_0>0\quad\tx{in}\quad\ol{\Om_L}.
\end{equation}
For a constant $d>0$, define
\begin{equation*}
\begin{split}
\mathscr{N}_d(\nabla\psi_0,\Phi_0,S_0,0):=\{ &(\bm q,z,S,\msK)\in [C^0(\ol{\Om_L})]^5:\\
&\|(\bm q, z, S, \msK)-(\nabla\psi_0,\Phi_0, S_0, 0)\|_{C^0(\ol{\Om_L})}<d\}.
\end{split}
\end{equation*}
 Applying the implicit mapping theorem yields that there exists a constant $d_0>0$ depending on the background data so that for any $(\bm q, z, S, \msK)\in \mcl{N}_{d_0}(\nabla\psi_0, \Phi_0, S_0, 0)$, there exists a unique $Q(\bm q, z, S, \msK)$  such that
\begin{equation}
\label{Q}
\mG(Q(\bm q, z, S, \msK), \bm q, z, S, \msK)=0\quad\tx{in}\quad \Om_L.
\end{equation}
Therefore,  if $(\nabla\psi, \Phi, S, \msK)\in \mathscr{N}_{d_0}$, then \eqref{1-e0} yields $\rho=Q(\nabla\psi, \Phi, S, \msK)$.

The transport equations for $S$ and $\msK$ in \eqref{1-f4t} can be solved in a  similar way. Furthermore, $S$ and $\msK$ paly the similar role in the system \eqref{1-a9} and \eqref{1-e5}. Therefore, to simplify the presentation, we assume
\begin{equation}
\label{2-a7-2}
\msB_{en}+\Phi_{bd}\equiv 0\quad\tx{on}\quad \Gamen
\end{equation}
so that $\msK=0$ and we focus on the study for the transport equation for $S$ later on.

For $\rx=(x_1,x_2)\in\Om_L$ and $(\bm q,z, S,0)\in \mathscr{N}_{d_0}(\nabla\psi_0,\Phi_0,S_0,0)$, define
\begin{equation}
\label{2-b1}
\begin{split}
&A_j(\rx, \bm q, z, S)=\frac{q_j}{Q(\bm q, z, S,0)}\quad\tx{for}\;\;j=1,2,\quad
B(\rx, \bm q, z, S)=Q(\bm q, z, S,0)
\end{split}
\end{equation}
where $(\bm q, z, S)$ are evaluated at $\rx$. From now on, we denote ${\bf A}(\rx,\bm q, z, S)=(A_1,A_2)(\rx,\bm q, z, S)$ and $B(\rx,\bm q, z, S)$ by

If $C\omega_1(G_{en}, \Phi_{bd}, \Sen, \msB_{en}, \pex)\le \frac{d_0}{2}$ in Theorem \ref{theorem-1}, then
it follows from  \eqref{Q} that \eqref{1-a9} and \eqref{1-e5} can be written as
\begin{equation}
\label{2-b2}
\left\{
\begin{split}
&\Div \left({\bf A}(\rx,\nabla\psi,\Phi, S)\right)=-\frac{B^{\gam}(\rx,\nabla\psi,\Phi,S)e^S S_{x_2} }{(\gam-1)\psi_{x_2}},\\
&\Delta\Phi =B(\rx,\nabla\psi,\Phi, S)
\end{split}
\right.
\end{equation}
There exist $d_1\in(0,d_0)$ and $C>0$ depending only on the background data such that if $(\bm q, z, S, 0)\in \mathscr{N}_{d_1}(\nabla\psi_0, \Phi_0, S_0, 0)$, then $({\bf A}, B)(\rx, \bm q, z, S)$ are continuously differentiable with respect to $(\bm q, z, S)$, and they satisfy
\begin{equation*}
|D_{(\bm q, z, S)}({\bf A}, B)(\rx,\bm q, z, S)|\le C.
\end{equation*}
Set
\begin{equation*}
(\phi,\Psi):=(\psi,\Phi)-(\psi_0,\Phi_0)\quad\tx{in}\quad\Om_L.
\end{equation*}
Suppose that $(\nabla\psi, \Phi, S, 0)\in \mathscr{N}_{d_1}(\nabla\psi_0, \Phi_0, S_0, 0)$.
Then $(\psi,\Phi)$ satisfy \eqref{2-b2} if and only if $(\phi,\Psi)$ satisfy
\begin{align}
\label{2-b5}
&\mclL_1(\phi,\Psi)=f(\rx,\nabla\phi,\Psi,S,S_{x_2})+\Div{\bm F}(\rx,\nabla\phi,\Psi, S),\\
\label{2-b6}
&\mclL_2(\phi,\Psi)=g(\rx,\nabla\phi,\Psi, S),
\end{align}
where $\mclL_1$, $\mclL_2$, $f$, ${\bm F}=(F_1,F_2)$ and $g$ are defined as follows:
\begin{align}
\label{2-b7}
&\mclL_1(\phi,\Psi)=\sum_{i=1}^2\der_i
\left(\sum_{j=1}^2\der_{q_j}A_i(\rx,\nabla\psi_0,\Phi_0,S_0)\der_j\phi
+\Psi\der_zA_i(\rx,\nabla\psi_0,\Phi_0,S_0)\right),\\
\label{2-b8}
&\mclL_2(\phi,\Psi)=\Delta\Psi-\Psi\der_zB(\rx,\nabla\psi_0,\Phi_0,S_0)
-\nabla\phi\cdot\der_{\bm q}B(\rx,\nabla\psi_0,\Phi_0,S_0),
\end{align}
\begin{equation}
\label{2-b10}
f(\rx,\bm q, z, S,S_{x_2})=\frac{(S-S_0)_{x_2}B^{\gam}(\rx,\nabla\psi_0+\bm q, \Phi_0+z,S)e^{S_0+(S-S_0)}}{(\gam-1)(m_0+q_2)},
\end{equation}
\begin{equation}
\label{2-b9}
\begin{split}
F_i(\rx,\bm q, z, S)=&\sum_{j=0}^2 q_j \int_0^1[\der_{q_j}A_i(\rx,\nabla\psi_0+\tau \bm q,\Phi_0+\tau z,S_0+\tau(S-S_0))]_{\tau=t}^0\;dt\\
&-(S-S_0)\int_0^1\der_{S}A_i(\rx,\nabla\psi_0+t \bm q,\Phi_0+t z,S_0+t(S-S_0))\;dt,
\end{split}
\end{equation}
\begin{equation}
\label{2-b11}
\begin{split}
g(\rx,\bm q, z, S)=&\sum_{j=0}^2 q_j \int_0^1[\der_{q_j}B(\rx,\nabla\psi_0+\tau \bm q,\Phi_0+\tau z,S_0+\tau(S-S_0))]_{\tau=0}^t\;dt\\
&+(S-S_0)\int_0^1\der_{S}B(\rx,\nabla\psi_0+t \bm q,\Phi_0+t z,S_0+t(S-S_0))\;dt
\end{split}
\end{equation}
with $q_0=z$.
Here $\der_1$ and $\der_2$ denote $\der_{x_1}$ and $\der_{x_2}$, respectively. And, $[k(\tau)]_{\tau=0}^t$ denotes $k(t)-k(0)$.

The transport equation for $S$ in \eqref{1-f4t} and the associated boundary condition at $\Gamen$ can be written as
\begin{equation}\label{eqS}
\nabla^\perp (\phi+\psi_0)\cdot \nabla S=0\,\, \text{in}\,\, \Omega_L
, \,\, S=S_{en}\,\, \text{on}\, \, \Gamma_0.
\end{equation}

Next, we compute boundary conditions for $(\phi, \Psi)$ corresponding to \eqref{1-f5}.
Solve \eqref{1-e0} for $|\nabla\psi|^2$ and substitute $(\rho, \Phi, \msK)=((\frac{p_{ex}}{e^S})^{1/\gam}, \Phi_{bd}, 0)$ given from \eqref{1-f5} so that we obtain the expression
\begin{equation}
\label{2-c5}
|\nabla\psi|^2=-2\left(\frac{p_{ex}}{e^S}\right)^{\frac{2}{\gam}}
\left(\Phi_{bd}-\frac{\gam e^{S/\gam}
p_{ex}^{1-\frac{1}{\gam}}}{\gam-1}\right)\quad\tx{on}\quad \Gamex.
\end{equation}
It follows from \eqref{1-e6} that \eqref{2-c5} can be written as a boundary condition for $\phi$ as follows:
\begin{equation}
\label{2-c6}
\phi_{x_2}=-\frac{1}{m_0}\left(h_1(\rx, S)+\frac 12 |\bm q|^2\right)=:h(\rx,\nabla\phi, S)\quad\tx{on}\quad\Gamex
\end{equation}
with $h_1(\rx, S)$ defined by
\begin{equation}
\label{2-c7}
\begin{split}
h_1(\rx, S)=\left(\frac{p_{ex}}{e^S}\right)^{\frac{2}{\gam}}
\left(\Phi_{bd}-\frac{\gam p_{ex}^{1-\frac{1}{\gam}}e^{S/\gam}}{\gam-1} \right)-
\left(\frac{\bar p(L)}{e^{S_0}}\right)^{\frac{2}{\gam}}
\left(\Phi_{0}-\frac{\gam e^{S_0/\gam}\bar p^{1-\frac{1}{\gam}}(L)}{\gam-1}\right).
\end{split}
\end{equation}
The rest of boundary conditions for $(\psi,\Phi)$ in \eqref{1-f5} is written in terms of $(\phi,\Psi)$ as follows:
\begin{equation}\label{2-d1bc}
\left\{
\begin{aligned}
&\Psi_{x_1}=G_{en}-G_0=:g_{en}\,\,\tx{on}\,\, \Gamen,\,\, \Psi=\Phi_{bd}-\Phi_0=:\Psi_{bd}\,\,\tx{on}\;\; \Gamex,\quad
\Psi_{x_2}=0\,\,\tx{on}\;\;\Gamw,\\
&\phi(0, -1) =0,\quad \phi_{x_1}=0\,\,\tx{on}\;\;\Gamen,\quad
\phi_{x_1}=0\,\,\tx{on}\;\;\Gamw.
\end{aligned}
\right.
\end{equation}
If $(\nabla\psi, \Phi, S, 0)\in \mathscr{N}_{d_1}(\nabla\psi_0, \Phi_0, S_0, 0)$ and $\psi_y \ge \frac{m_0}{10}$ in $\ol{\Om}_L$, then $(\psi, \Phi, S, 0)$ solve \eqref{1-a9},  \eqref{1-e5}, \eqref{1-f4t}, and \eqref{1-f5} if and only if $(\phi, \Psi, S)$ solve \eqref{2-b5}, \eqref{2-b6}, \eqref{eqS}, and \eqref{2-d1bc}. To prove Theorem \ref{theorem-1}, it suffices to solve the nonlinear boundary value problem \eqref{2-b5}, \eqref{2-b6}, \eqref{eqS}, and \eqref{2-d1bc} for $(\phi, \Psi, S)$ for $C\omega_1(G_{en}, \Phi_{bd}, S_{en}, \msB_{en}, \pex)\le \min(\frac{d_1}{2}, \frac{m_0}{2})$.

Given $S$, the equations \eqref{2-b7} and \eqref{2-b8} form a nonlinear elliptic system for $(\phi, \Psi)$ provided that $\|\phi\|_{C^1(\ol{\Om_L})}+\|\Psi\|_{C^0(\ol{\Om_L})}$ is sufficiently small. Given $\phi$ with $\|\phi\|_{C^1(\ol{\Om_L})}$ small, \eqref{eqS} can be regarded as an initial value problem for $S$.
Based on this observation, we prove solvability of \eqref{2-b5}, \eqref{2-b6}, \eqref{eqS}, \eqref{2-d1bc} by the method of iteration.
For that purpose, it is crucial
to study the boundary value problem for a linear elliptic system associated with \eqref{2-b5}, \eqref{2-b6}, and \eqref{2-d1bc}.
\section{Linear boundary value problem}
\label{section-3}
Fix $\alp\in(0,1)$. Given $\mff\in C^{\alp}(\ol{\Om_L})$, $\mfg\in C^{\alp}(\ol{\Om_L})$, $g_{en}\in C^{\alp}(\ol{\Gamen})$, $\mfh\in C^{\alp}(\ol{\Gamex})$ and ${\bf \mfF}\in (C^{1,\alp}_{(-\alp,\corners)}({\Om_L}))^2$,
consider the
linear boundary value problem,
\begin{align}
\label{3-a1}
&\begin{cases}
\mclL_1(\phi,\Psi)=\mff+\Div {\bf \mfF}\\
\mclL_2(\phi,\Psi)=\mfg
\end{cases}\quad \tx{in}\quad\Om_L
\end{align}
with boundary conditions
\begin{align}
\label{3-a1-b1}
&\Psi_{x_1}=g_{en}\,\,\tx{on}\quad \Gamen,\,\, \Psi=\Psi_{bd}\,\,\tx{on}\quad \Gamex,\quad
\Psi_{x_2}=0\,\,\tx{on}\,\,\Gamw,\\
\label{3-a1-b2}
&\phi(0,-1)=0,\,\, \phi_{x_1}=0\,\,\tx{on}\,\, \Gamen,\,\,
\phi_{x_1}=0\,\,\tx{on}\,\, \Gamw,\,\,
\phi_{x_2}=\mfh\,\,\tx{on}\,\, \Gamex.
\end{align}
If $\phi\in C^1(\ol{\Om_L})$, then the boundary conditions \eqref{3-a1-b2} are equivalent to
\begin{equation}
\label{3-b3}
\begin{cases}
\phi_{x_1}=0&\tx{on}\quad \Gamen,\\
\phi(x_1,-1)=0,\quad \phi(x_1,1)=\int_{-1}^1\mfh(z)\;dz&\tx{on}\quad \Gamw,\\
\phi(L,x_2)=\int_{-1}^{x_2} \mfh(z)\;dz&\tx{on}\quad \Gamex.
\end{cases}
\end{equation}
So we consider the boundary value problem \eqref{3-a1}, \eqref{3-a1-b1} and \eqref{3-b3} in the rest of this section.

\begin{proposition}\label{proposition-2}
Let $\mclL_1$ and  $\mclL_2$ be defined by \eqref{2-b7} and \eqref{2-b8} for a background solution $(\bar{\rho}, \bar{\bf u}, \bar p, \Phi_0)$
satisfying the assumptions in Theorem \ref{theorem-1}.
Then there exists a unique solution $(\phi,\Psi)\in [C^{1,\alp}(\ol{\Om_L})\cap C^2(\Om_L)]^2$ to
 \eqref{3-a1}, \eqref{3-a1-b1} and \eqref{3-b3} satisfying
 \begin{equation}
 \label{3-e3}
 \begin{aligned}
 \|(\phi,\Psi)\|_{2,\alp,\Om_L}^{(-1-\alp,\corners)}\le& C\left(\|(\mff,\mfg)\|_{\alp,\Om_L}+\|{\bm \mfF}\|_{1,\alp,\Om_L}^{(-\alp,\corners)}+
 \|g_{en}\|_{1,\alp,\Gamen}^{(-1-\alp,\der\Gamen)} \right.\\
 &\quad \,\, \left.+\|\Psi_{bd}\|_{2,\alp,\Om_L}^{(-1-\alp,\corners)}
 +\|\mfh\|_{1,\alp,\Gamex}^{(-\alp,\partial\Gamw)}
 \right)
 \end{aligned}
 \end{equation}
 where the constant $C$ depends only on the background data and $\alp$.

\end{proposition}


For $\rx=(x_1,x_2)\in \Om_L$, define
\begin{equation*}
\Psi_{bd}^*(\rx):=\Psi_{bd}(L,x_2),\,\,
\quad
\,\,
\mfH(x_1,x_2)=\int_{-1}^{x_2} \mfh(z)\;dz.
\end{equation*}
If $\Psi_{bd}\in C^{2,\alp}_{(-1-\alp,\der\Gamex)}(\Gamex)$ satisfies \eqref{1-f10},
then one has
\begin{equation*}
(\Psi-\Psi_{bd}^*)_{x_1}=g_{en}\,\,\tx{on}\quad \Gamen,\quad \Psi-\Psi_{bd}^*=0\,\, \text{on} \,\,\Gamex,\quad
\der_{x_2}(\Psi-\Psi_{bd}^*)=0\,\,\tx{on}\quad\Gam_w.
\end{equation*}
Set
\begin{equation}
\label{3-b4}
\mfu(x_1,x_2):=\phi(x_1,x_2)- \mfH(x_1, x_2),\quad \mfV(x_1,x_2):=\Psi(x_1,x_2)-\Psi_{bd}^*(x_1,x_2).
\end{equation}
To simplify notations, we write
\begin{equation}
\label{3-a2}
\begin{split}
&\mfraka_{ij}(\rx)=\der_{q_j}A_i(\rx,\nabla\psi_0,\Phi_0, S_0),\quad
\mfrakb_i(\rx)=\der_zA_i(\rx,\nabla\psi_0,\Phi_0, S_0)\\
&\mfrakc_i(\rx)=\der_{q_i}B(\rx,\nabla\psi_0,\Phi_0, S_0),\quad\quad
\mfrakd(\rx)=\der_zB(\rx,\nabla\psi_0,\Phi_0, S_0)
\end{split}
\end{equation}
for $i,j=1,2$.
Note that $\mfraka_{ij}, \mfrakb_i, \mfrakc_i$ and $\mfrakd$ are independent of $x_2$ for $\rx=(x_1,x_2)\in\Om_L$.

Then
$(\phi,\Psi)$ solve \eqref{3-a1}, \eqref{3-a1-b1} and \eqref{3-b3} if and only if $(\mfu, \mfV)$ solve
\begin{equation}
\label{3-b5}
\begin{split}
&\begin{cases}
\mclL_1(\mfu,\mfV)=\mff+\Div {\bm \mfF}-\mclL_1(\mfH,\Psi_{bd}^*)=:\mff+\Div{\bm \mfF}^*\\
\mclL_2(\mfu,\mfV)=\mfg-\mclL_2(\mfH,\Psi_{bd}^*)=:\mfg^*+\Div{\bf \mfG}^*
\end{cases}\quad \tx{in}\quad\Om_L
\end{split}
\end{equation}
with 
boundary conditions
\begin{align}
\label{3-b5bc1}
&\mfV_{x_1}=g_{en}\,\,\tx{on}\quad \Gamen,\quad \mfV=0\,\,\tx{on}\quad \Gamex,\quad \mfV_{x_2}=0\,\,\tx{on}\quad \Gamw,\\
\label{3-b5bc2}
&\mfu_{x_1}=0\,\,\tx{on}\quad \Gamen,\quad \mfu=0\,\,\tx{on}\quad \Gamw\cup\Gamex,
\end{align}
where
\begin{equation}
\label{3-b6}
\begin{split}
&{\bf \mfF}^*={\bf \mfF}-(0,{\mfraka}_{22}\der_2\mfH)-(\mfrakb_1,\mfrakb_2)\Psi^*_{bd},\quad
{\bf \mfG}^*=-\nabla\Psi^*_{bd}, \quad \mfg^*=\mfg+\Psi^*_{bd}\mfrakd+\mfrakc_2\der_2\mfH.
\end{split}
\end{equation}
To prove Proposition \ref{proposition-2}, it suffices to show
that \eqref{3-b5}-\eqref{3-b5bc2} is uniquely solvable.

\begin{lemma}
\label{lemma-1}
Let $\mfraka_{ij}$, $\mfrakb_i$, $\mfrakc_i$, and $\mfrakd$ be defined
by \eqref{3-a2} for a background solution $(\bar{\rho}, \bar{\bf u}, \bar p, \Phi_0)$
satisfying the assumptions in Theorem \ref{theorem-1}. Then, they
satisfy the following properties:
\begin{itemize}
\item[(a)] The matrix $[\mfraka_{ij}]_{i,j=1}^2$ is diagonal, and there exists a constant $\lambda_0>0$ satisfying
\begin{equation}
\label{3-a3}
\lambda_0|{\bm \xi}|^2\le \sum_{i,j=1}^2\mfraka_{ij}(\rx)\xi_i\xi_j\le \frac{1}{\lambda_0}|{\bm \xi}|^2\quad \text{for all}\,\,\rx\in \ol{\Om_L}\,\,\text{and}\,\, {\bm \xi}=(\xi_1,\xi_2)\in\R^2.
\end{equation}
\item[(b)] For each $k\in\mathbb{Z}_+$, there exists a constant $\mcl{C}_k$ satisfying
   \begin{equation*}
   \sum_{i,j=1}^2\|\mfraka_{ij}\|_{k,\Om_L}
   +\sum_{i=1}^2\|(\mfrakb_i,\mfrakc_i)\|_{k,\Om_L}
   +\|\mfrakd\|_{k,\Om_L}\le \mcl C_k.
    \end{equation*}
\item[(c)] For each $i=1,2$, we have
\begin{equation*}
\mfrakb_i(\rx)+\mfrakc_i(\rx)=0\quad \text{in}\,\, \Omega_L.
\end{equation*}

\item[(d)] 
For $\delta_0\in(0,1)$ from \eqref{estd0}, $\mfrakd$ satisfies
\begin{equation*}
-\frac{2}{L^2}(1-\delta_0) \le \mfrakd<0
\quad\tx{in}\quad \Om_L.
\end{equation*}
\end{itemize}
Here, $\lambda_0$ depends on the data, and $\mcl{C}_k$ depends on $k$ and the data.
\begin{proof}
The direct computations using \eqref{1-e8}, \eqref{2-b1} and \eqref{3-a2} give
\begin{equation}
\label{3-a7}
\begin{split}
&\mfraka_{ij}=\frac{1}{\gam e^{S_0}\bar\rho^{\gam}-\frac{m_0^2}{\bar\rho}}
\left((\gam e^{S_0}\bar\rho^{\gam-1}-\frac{m_0^2}{\bar\rho^2})\delta_{ij}
+\frac{m_0^2\delta_{i2}\delta_{j2}}{\bar\rho^2}\right),\\
&\mfrakb_i=-\mfrakc_i=\frac{m_0\delta_{i2}}{\gam e^{S_0}\bar\rho^{\gam}-\frac{m_0^2}{\bar\rho}},\quad
\mfrakd=\frac{-\bar\rho}{\gam e^{S_0}\bar\rho^{\gam-1}-\frac{m_0^2}{\bar\rho^2}}
\end{split}
\end{equation}
for $i,j=1,2$, where $\delta_{ij}$ is the Kronecker symbol satisfying $\delta_{ij}=1$ for $i=j$, $\delta_{ij}=0$ for $i\neq j$. Then (a)--(d) easily follow from \eqref{1-b9}, \eqref{estd0}, \eqref{3-a7} and Proposition \ref{proposition-1}.
\end{proof}
\end{lemma}

Define $\mcl{H}= \{(\zeta,\omega)\in [H^1(\Om_L)]^2:\zeta=0\;\tx{on}\;\Gamw\cup\Gamex, \omega=0\;\tx{on}\;\Gamex\}$, which is a Hilbert space.
If $(\mfu,\mfV)\in\mcl{H}$ satisfy
\begin{equation}
\label{3-c1}
\mfrakL[(\mfu,\mfV),(\zeta,\omega)]=\langle(\mff,{\bm \mfF}^*,\mfg^*,{\bf \mfG}^*, g_{en}),(\zeta,\omega)\rangle
\end{equation}
for all $(\zeta,\omega)\in \mcl{H}$, where $\mfrakL$ and $\langle\cdot,\cdot\rangle$ are defined as follows:
\begin{equation}
\label{3-b7}
\mfrakL[(\mfu,\mfV),(\zeta,\omega)]
=\int_{\Om_L}\sum_{i=1}^2(\mfraka_{ii}\der_i\mfu+\mfrakb_i\mfV)\der_i\zeta+\nabla\mfV\cdot\nabla\omega
+(\mfrakd\mfV+ \sum_{i=1}^2\mfrakc_i \partial_i\mfu)\omega\;d\rx
\end{equation}
and
\begin{equation*}
\langle(\mff,{\bm \mfF}^*,\mfg^*,{\bf \mfG}^*),(\zeta,\omega)\rangle=
\int_{\Om_L}{\bm \mfF}^*\cdot\nabla\zeta+{\bf \mfG}^*\cdot\nabla\omega\;d\rx-
\int_{\Om_L}\mff\zeta+\mfg^*\omega\;d\rx+\int_{\Gamen}({\mfF}_1^*\zeta-g_{en}\omega) \;dx_2,
\end{equation*}
  then  we say that $(\mfu,\mfV)$ is a weak solution of \eqref{3-b5}--\eqref{3-b5bc2}.
If $(\mfu,\mfV)\in [C^1(\ol{\Om_L})\cap C^2(\Om_L)]^2$ solve \eqref{3-b5}, then they must be a weak solution.
\begin{lemma}
\label{lemma-3}
Let $\mfraka_{ij}$, $\mfrakb_i$, $\mfrakc_i$, and $\mfrakd$ be defined
by \eqref{3-a2} for a background solution $(\bar{\rho}, \bar{\bf u}, \bar p, \Phi_0)$
satisfying the assumptions in Theorem \ref{theorem-1}.
Then \eqref{3-b5}--\eqref{3-b5bc2} has a  unique weak solution $(\mfu,\mfV)\in\mcl{H}$ satisfying the estimate
\begin{equation}
\label{3-c2}
\|(\mfu,\mfV)\|_{H^1(\Om_L)}\le C(\|(\mff,  \mfg^*,{\bf \mfG}^*, {\bm\mfF}^*\|_{C^0(\ol{\Om_L})}+\|g_{en}\|_{C^0(\ol{\Gamen})} )
\end{equation}
where the constant $C$ depends only on the data and $\delta_0$.
\begin{proof}
By Lemma \ref{lemma-1}(c),
\begin{equation*}
\mfrakL[(\zeta,\omega),(\zeta,\omega)]=\int_{\Om_L}\sum_{i=1}^2\mfraka_{ii}(\der_i\zeta)^2+|\nabla\omega|^2
+\mfrakd\omega^2\;d\rx
\end{equation*}
holds for all $(\zeta,\omega)\in \mcl{H}$. For each $w\in H^1({\Om_L})$ with $w=0$ on $\Gamen$, fundamental theorem of calculus and $C^1$-approximations of $H^1$ functions
give
\begin{equation}
\label{p-ineq}
\int_{\Om_L} w^2\;d\rx\le \frac 12 L^2\int_{\Om_L} w_{x_1}^2\;d\rx
\end{equation}
from which we get
\begin{equation*}
\mfrakL[(\zeta,\omega),(\zeta,\omega)]\ge \int_{\Om_L}
\sum_{i=1}^2\mfraka_{ii}(\der_i\zeta)^2+(1
-\frac{L^2}{2}\|\mfrakd\|_{C^0(\ol{\Om_L})})|\nabla\omega|^2\;d\rx.
\end{equation*}
It follows from \eqref{3-a3} and \eqref{estd0}  that
\begin{equation}\label{poincareest}
\mfrakL[(\zeta,\omega),(\zeta,\omega)]\ge
\int_{\Om_L}\lambda_0|\nabla\zeta|^2+{ \delta_0|\nabla\omega|^2\;d\rx} \quad\tx{for all}\;\;(\zeta,\omega)\in\mcl{H}.
\end{equation}
Combining \eqref{poincareest}
with Poincar\'{e} inequality yields
\begin{equation*}
\mfrakL[(\zeta,\omega),(\zeta,\omega)]\ge C \|(\zeta, \omega)\|^2_{H^1(\Omega_L)} \quad\tx{for all}\;\;(\zeta,\omega)\in\mcl{H}
\end{equation*}
for a constant $C$ depending only on the data and $\delta_0$.
This implies that the bilinear operator $\mfrakL:\mcl{H}\times \mcl{H}\rightarrow \R$ is coercive. Furthermore, it follows from H\"{o}lder inequality, trace inequality, and Poincar\'{e} inequality that 
\begin{equation*}
|\langle(\mff,{\bf \mfF}^*, \mfg^*,{\bf \mfG}^*, g_{en}),(\zeta,\omega)\rangle|\le \hat{C}\|(\zeta,\omega)\|_{H^1(\Om)}
(\|(\mff,  \mfg^*,{\bf \mfG}^*, {\bm\mfF}^*\|_{C^0(\ol{\Om_L})}+\|g_{en}\|_{C^0(\ol{\Gamen})} )
\end{equation*}
holds where the constant $\hat{C}>0$ is independent of the data.
 Hence we can apply Lax-Milgram theorem to  \eqref{3-c1} to conclude that, for any given $(\mff,  \mfg^*,{\bf \mfG}^*, {\bf \mfF}^*, g_{en})\in[C^0(\ol{\Om_L})]^6\times C^0(\ol{\Gamen})$, there exists a unique $(\mfu,\mfV) \in \mcl{H}$ satisfying \eqref{3-c1}. Furthermore such $(\mfu,\mfV)$ satisfy
\eqref{3-c2}.
\end{proof}
\end{lemma}


\begin{lemma}
\label{lemma-5}
For any $\alp\in(0,1)$, there exists a constant  $C_B>0$ depending only on the background data, $\delta_0$ and $\alp$ such that whenever $({\bf\mfF}^*,{\bf\mfG}^*, g_{en}) \in [C^{\alp}(\ol{\Om_L},\R^2)]^4\times C^{\alp}(\ol{\Gamen})$,
the weak solution $(\mfu,\mfV)\in\mcl{H}$ of \eqref{3-b5}--\eqref{3-b5bc2}
satisfy
\begin{equation}
\label{3-c7}
\|(\mfu,\mfV)\|_{1,\alp,\Om_L}\le C_B\mcl{M}(\mff,{\bf \mfF^*},\mfg^*, {\bf \mfG}^*, g_{en})
\end{equation}
for
\begin{equation}
\label{3-c7-2}
\mcl{M}(\mff,{\bf \mfF^*},\mfg^*, {\bf \mfG}^*, g_{en})=\|(\mff,\mfg^*)\|_{L^{\infty}(\Om_L)}+
\|({\bf \mfF}^*,{\bf \mfG}^*)\|_{\alp,\Om_L}+ \|g_{en}\|_{\alp,\Gamen}.
\end{equation}
\begin{proof}
This lemma can be proved by adjusting arguments in \cite{BDX}, so we outline the idea of proof. For details, one can refer to \cite[Lemmas 3.5 and 3.6]{BDX}.

Step 1. One can adjust the proof of \cite[Lemma 3.5]{BDX} to find constants $R\in(0,\frac{\min\{1,L\}}{10}]$ and $C>0$ depending only the background data, $\delta_0$ and $\alp$ so that
\begin{equation}
\label{3-d1}
\int_{B_r(\rx_0)\cap\Om_L}|\nabla\mfu|^2+|\nabla\mfV|^2\;d\rx\le Cr^{2\alp}[\mcl{M}(\mff,{\bf \mfF}^*, \mfg^*, {\bf \mfG}^*, g_{en})]^2
\end{equation}
holds for all $\rx_0\in\ol{\Om_L}$ and $r\in[0,R]$.
Combining \eqref{3-d1} with Lemma \ref{lemma-3} gives
\begin{equation}
\label{3-d2}
\|\mfu\|_{\alp,\Om_L}+\|\mfV\|_{\alp,\Om_L}\le C\mcl{M}(\mff,{\bm \mfF^*},\mfg^*, {\bf \mfG}^*, g_{en}).
\end{equation}

Step 2. Substituting $\omega=0$ into \eqref{3-b7}, $\mfu\in \{\zeta\in H^1(\Om_L):\zeta|_{\Gamw\cup \Gamen}=0\}(=:\mcl{H}_1)$ is regarded as a solution to
\begin{equation}
\label{3-d10}
\int_{\Om_L}\sum_{i=1}^2\mfraka_{ii}\der_i\mfu\der_i\zeta\;d\rx=
\int_{\Om_L}{\bf \mfF}^{\sharp}\cdot\nabla\zeta-\mff\zeta\;d\rx+\int_{\Gamen}{\bf \mfF}_1^{\sharp}\zeta \;dx_2\;\;
\tx{for all}\;\;\zeta\in\mcl{H}_1
\end{equation}
where we set
\begin{equation}
\label{Fsharp}
{\bf \mfF}^{\sharp} =(\mfF^\sharp_1, \mfF^\sharp_2)={\bf \mfF}^*-({\mfrakb}_1,\mfrakb_2)\mfV.
\end{equation}
We use this expression, \eqref{3-d2} and adjust the proofs of \cite[Lemma 1.41, Theorems 3.1 and 3.13]{Ha-L} to obtain
\begin{equation}
\label{3-e1}
\|\mfu\|_{1,\alp,\Om_L}\le C \mcl M(\mff,{\bf \mfF}^*, \mfg^*, {\bf \mfG}^*, g_{en})
\end{equation}
where the constant $C$ is chosen depending only on the background data, $\delta_0$ and $\alp$.

Step 3. Substituting $\zeta=0$ into \eqref{3-b7}, $\mfV\in \{\omega\in H^1(\Om):\omega|_{\Gamex}=0\}(=:\mcl{H}_2)$ is regarded as a solution to
\begin{equation*}
\int_{\Om} \nabla\mfV\cdot\nabla\omega\;d\rx=
\int_{\Om} {\bf \mfG}^*\cdot\nabla\omega-\til \mfg^*\omega\;d\rx-\int_{\Gamen} g_{en}\omega\;dx_2\;\;\tx{for all}\;\;\omega\in \mcl{H}_2
\end{equation*}
where we set
$\til \mfg^*=\mfg^*+(\mfrakc_1,\mfrakc_2)\cdot \nabla\mfu+\mfrakd \mfV$. Combining Lemma \ref{lemma-1}, \eqref{3-d2} and \eqref{3-e1} yields
\begin{equation*}
\|\til \mfg^*\|_{L^{\infty}(\Om_L)}\le C \mcl M(\mff,{\bf \mfF}^*, \mfg^*, {\bf \mfG}^*, g_{en}).
\end{equation*}
With the help of the compatibility condition \eqref{1-f10}, we can extend $\mfV$ evenly with respect to the insulated boundary near the corner points $\bar\Gamw\cap\bar\Gamex$.
Adjusting the proofs of \cite[Lemma 1.41, Theorems 3.1 and 3.13]{Ha-L} again, we obtain
\begin{equation*}
\|\mfV\|_{1,\alp,\Om_L}\le C\mcl  \mcl M(\mff,{\bf \mfF}^*, \mfg^*, {\bf \mfG}^*, g_{en}).
\end{equation*}

\end{proof}
\end{lemma}

\begin{remark}
\label{remark-1}
According to \cite[Theorem 3.13]{Ha-L}, for $\alpha\in (0, 1)$,  Lemma \ref{lemma-5} is still valid when the definition of $ \mcl M(\mff,{\bf \mfF}^*, \mfg^*, {\bf \mfG}^*, g_{en})$ in \eqref{3-c7-2} is replaced by
\begin{equation*}
 \mcl M(\mff,{\bf \mfF}^*, \mfg^*, {\bf \mfG}^*, g_{en})=
\|(\mff,\mfg^*, g_{en})\|_{L^{q}(\Om_L)}+
\|({\bf \mfF}^*, {\bf \mfG}^*)\|_{\alp,\Om_L}+\|g_{en}\|_{\alp,\Gamen}
\end{equation*}
with $q=\frac{2}{1-\alp}$. This will be used in Section \ref{section-4}.
\end{remark}
Now we are ready to prove Proposition \ref{proposition-2}.
\begin{proof}[Proof of Proposition \ref{proposition-2}]
Given $(\mff, \mfg,\mfh, {\bf \mfF})$, let $\mfF^*, \mfG^*, \mfg^*$ be given by \eqref{3-b6}.  Then, Lemmas \ref{lemma-3} and \ref{lemma-5} imply that there exists unique weak solution $(\mfu, \mfV)\in \mcl{H}$ of the problem \eqref{3-b5}--\eqref{3-b5bc2} which satisfies \eqref{3-c7}.
It follows from \eqref{3-b6} that ${\mfF}^{\sharp}$ given by \eqref{Fsharp} satisfies the estimate
\begin{equation}
\label{3-e5}
\|{\mfF}^{\sharp}\|_{1,\alp,\Om_L}^{(-\alp,\corners)}\le C\left(\|{\mfF}\|_{1,\alp,\Om_L}^{(-\alp,\corners)}+\|\mfh\|_{1,\alp,\Gamex}+
\|\Psi_{bd}\|_{2,\alp,\Gamex}^{(-1-\alp,\der\Gamex)}+\|g_{en}\|_{\alp,\Gamen}\right).
\end{equation}
According to \eqref{3-d10}, $\mfu$ is a weak solution of the problem
\begin{equation}
\label{3-e4}
\left\{
\begin{aligned}
&\sum_{i=1}^2\der_i(\mfraka_{ii}\der_i\mfu)=\Div{\bm \mfF}^{\sharp}+\mff\quad\tx{in}\quad\Om_L\\
&\mfu=0\quad\tx{on}\quad \der\Om_L\setminus \Gamen,\qquad\mfu_{x_1}=0\quad\tx{on}\quad \Gamex.
\end{aligned}
\right.
\end{equation}
Therefore, we can follow the argument of the proof of  \cite[Proposition 3.1]{BDX} to conclude that
$\mfu \in C^2(\Om_L)$ and $\mfu$ satisfies \eqref{3-e4} pointwisely in $\ol{\Om_L}\setminus \corners$.
For any fixed $\rx_0\in \ol{\Om_L}\setminus \corners$, define $d=\frac{1}{2} {\rm dist}(\rx_0,\corners)$ and a scaled function
\begin{equation*}
{\mfu}^{(\rx_0)}({\rm y}):=\frac{1}{d^{1+\alp}}
\left(\mfu(\rx_0+d{\rm y})-\mfu(\rx_0)-d\nabla\mfu(\rx_0)\cdot {\rm y}\right)
\end{equation*}
for ${\rm y}\in\{{\rm y}\in B_1(\bm 0):\rx_0+d{\rm y}\in\Om_L\}=:\mathfrak{B}_1(\rx_0)$. Applying the standard elliptic estimate \cite{GT} and  Lemma \ref{lemma-5} yields
\begin{equation*}
\|\mfu^{(\rx_0)}\|_{2,\alp,\mathfrak{B}_{1/2}(\rx_0)}\le C\left(\|{ \mfF^{\sharp}}\|_{1,\alp,\Om_L}^{(-\alp,\corners)}
+\|\mff\|_{\alp,\Om_L}\right).
\end{equation*}
This, together with \eqref{3-e5}, gives
\begin{equation}
\label{3-e6}
\|\mfu\|_{2,\alp,\Om_L}^{(-1-\alp,\corners)}\le C\left(\|{\bm \mfF}\|_{1,\alp,\Om_L}^{(-\alp,\corners)}+\|\mfh\|_{1,\alp,\Gamex}+
\|\Psi_{bd}\|_{2,\alp,\Gamex}^{(-1-\alp,\der\Gamex)}+\|g_{en}\|_{\alp,\Gamen}\right).
\end{equation}
Furthermore, it follows from \eqref{3-b6} and \eqref{3-e6} that
\begin{equation}
\label{3-e7}
\begin{split}
&\|\mfV\|_{2,\alp,\Om}^{(-1-\alp,\corners)}\\
&\le C\left(
\|{\bm \mfF}\|_{1,\alp,\Om}^{(-\alp,\corners)}
+\|\mfh\|_{1,\alp,\Gamex}+
\|\Psi_{bd}\|_{2,\alp,\Om}^{(-1-\alp,\corners)}
+\|\mfg\|_{\alp,\Om}+\|g_{en}\|_{1,\alp,\Gamen}^{(-\alp,\der\Gamen)}\right).
\end{split}
\end{equation}
The estimate constant $C$ in \eqref{3-e6} and \eqref{3-e7} depends only on the background data, $\delta_0$
 and $\alp$.

Finally,
 \eqref{3-e3}
 easily follows from \eqref{3-b4}, \eqref{3-e6} and \eqref{3-e7}.
\end{proof}

\section{Existence and uniqueness for nonlinear problems}\label{section-4}
We prove Theorem \ref{theorem-1}
by Schauder fixed point theorem. Since Eq.\eqref{2-b5} is coupled with
the transport equation \eqref{eqS}  through a derivative of $S$,
the uniqueness is proved under additional condition \eqref{1-f13}.

\medskip

Fix $\alpha\in (0, 1)$.
Given a constant $\delta>0$ to be determined later,  define
\begin{equation*}
\begin{aligned}
\mcl{K}_\delta:= \{(\phi, \Psi)\in [C^{1,\alp}(\ol{\Om_L})]^2:\phi(x_1,-1)=0,\; \der_{x_1}\phi(x_1,1)=0,\\ \|(\phi,\Psi)\|_{2,\alp,\Om_L}^{(-1-\alp,\corners)}\le \delta \}.
\end{aligned}
\end{equation*}
\begin{lemma}
\label{lemma-6}
Suppose that $\Sen\in C^{1,\alp}(\ol{\Gamen})$. There exist a constant $\delta_1>0$ such that if $\delta \le \delta_1$ and $(\til{\phi}, \til\Psi)\in \mcl{K}_\delta$, then the problem
\begin{equation}
\label{4-a5}
\nabla^\perp (\psi_0+\til\phi) \cdot \nabla S=0\quad \text{in}\,\, \Omega_L,\quad S=S_{en}\quad \text{on}\,\, \Gamma_0
\end{equation}
has a unique solution $S\in C^{1,\alp}(\ol{\Om_L})$. Furthermore,  $S$ satisfies
\begin{equation}
\label{4-a2}
\|S-S_0\|_{1,\alp,\Om_L}\le C\|S_{en}-S_0\|_{1,\alp,\Gamen}
\end{equation}
for a constant $C>0$ depending only on the background data.
\begin{proof}
Fix $(\til{\phi}, \til{\Psi})\in \mcl{K}_{\delta}$ and set $\til\psi =\psi_0+\til\phi$.
Let $\delta_1 = \frac{m_0}{4}$. If $\delta \in (0, \delta_1)$, then
\begin{equation}
\label{2-a1}
\frac 34 m_0\le \til\psi_{x_2}\le \frac 54 m_0\quad\tx{in}\quad \ol{\Om_L}.
\end{equation}
The boundary condition $\til\psi_{x_1}(x_1,\pm1)=0$ combined with \eqref{2-a1} yields
\begin{equation*}
\til\psi(0,-1)\le\til\psi(x_1,x_2)\le \til\psi(0,1)\quad\tx{in}\quad\ol{\Om_L}.
\end{equation*}
Set $I_{\til{\psi}}:=[\til{\psi}(0,-1),\; \til{\psi}(0,1)]$.
The implicit function theorem implies that for any $x_1\in[0,L]$ and $\xi\in I_{\til{\psi}}$, there exists a  unique $\pi (x_1,\xi)\in[-1,1]$ satisfying
\begin{equation*}
\til\psi(x_1,\pi(x_1,\xi))=\xi,
\end{equation*}
and $\pi(x_1,\xi)$ is continuously differentiable with respect to $x_1$ and $\xi$.  Therefore, for any $(x_1,x_2)\in\ol{\Om_L}$, there exists a unique $\theta \in[-1,1]$ such that
\begin{equation}
\label{2-a4}
\til\psi(x_1,x_2)=\til\psi(0,\theta)
\end{equation}
holds. Note that $\mathscr{F}_{(\til\psi)}(\theta):=\til\psi(0,\theta)$ is an invertible function from $[-1,1]$ onto $I_{\til{\psi}}$. Thus \eqref{2-a4} gives
\begin{equation}
\label{2-a5}
\theta(x_1,x_2)=\mathscr{F}^{-1}_{(\til\psi)}\circ \til\psi(x_1,x_2)=:
\mathscr{L}^{(\til\psi)}(x_1,x_2)\quad\tx{in}\quad\ol{\Om_L}
\end{equation}
and
\begin{equation}
\label{2-a7}
\nabla\mathscr{L}^{(\til\psi)}(x_1,x_2)=
\frac{\nabla\til\psi(x_1,x_2)}{\til\psi_{x_2}(0,\mathscr{L}^{(\til\psi)}(x_1,x_2))}.
\end{equation}
It follows from \eqref{4-a5} and \eqref{2-a5} that $S$ is  given by
\begin{equation*}
S(x_1,x_2)=S_{en}\circ\mathscr{L}^{(\til\psi)}(x_1,x_2).
\end{equation*}
 It is easy to see that $\mathscr{L}^{(\til{\psi})}$ is well defined by \eqref{2-a5} for all $(\til{\phi}, \til\Psi)\in \mcl{K}_{\delta}$ whenever $0<\delta \le \delta_1$. Moreover, by the choice of $\delta_1$ and \eqref{2-a7}, $\mathscr{L}^{(\til{\psi})}$ satisfies the estimate
\begin{equation}
\label{4-a4}
\|\mathscr{L}^{(\til{\psi})}\|_{1,\alp,\Om_L}\le Cm_0.
\end{equation}
 Since $S_0$ is a constant, $S-S_0$ can be written as
$
S-S_0=(\Sen-S_0)\circ \mathscr{L}^{(\til{\psi})}$ in $\Om_L$. Thus the estimate \eqref{4-a2} is a direct consequence of \eqref{4-a4}.
\end{proof}
\end{lemma}

Now we are in position to prove Theorem \ref{theorem-1}.

{\textbf{Proof of Theorem \ref{theorem-1} (a):}}
If $\delta \in (0, \frac{m_0}{4})$, then for any given $(\til{\phi},\til{\Psi})\in \mcl{K}_\delta$, the problem \eqref{4-a5} has a unique solution in $C^{1,\alp}(\ol{\Om_L})$.
Let $\til S$ be
the solution of \eqref{4-a5}. 
there exists a constant $\delta_2>0$ depending on the background data and $\alp$ such that if $\max\{\delta, \|S_{en}-S_0\|_{1,\alp,\Gamen}\}\le \delta_2$,
then
\begin{equation}
\label{4-a6}
\begin{split}
&\til{\bf F}={\bf F}(\rx,\nabla\til{\phi}, \til{\Psi},\til S),\quad \til f=f(\rx,\nabla\til{\phi}, \til{\Psi},\til S, \til S_{x_2}),\\
&\til g=g(\rx,\nabla\til{\phi}, \til{\Psi},\til S),\quad\;
\til h=h(\rx, \nabla\til{\phi}, \til S)
\end{split}
\end{equation}
are well defined in $\Om_L$ where ${\bf F}$, $f$, $g$ and $h$  are defined by \eqref{2-b10}--\eqref{2-b11} and \eqref{2-c6}. Moreover, there exists a constant $C$ depending only on the background data and $\alpha$ such that
\begin{equation}
\label{4-a7}
\begin{split}
&\|\til{\bf F}\|_{1,\alp,\Om_L}^{(-\alp,\corners)}
+\|\til f\|_{\alp,\Om_L}+\|\til g\|_{\alp,\Om_L}
\le C\left(\|\Sen-S_0\|_{1,\alp,\Gamen}+\delta\|(\til{\phi},  \til{\Psi})\|_{2,\alp,\Om_L}^{(-1-\alp,\corners)}\right),\\
&\|\til h\|_{1,\alp,\Gamex}^{(-\alp,\partial\Gamex)}\le
C\left(\|\pex-\bar p(L)\|_{1,\alp,\Gamex}^{(-\alp,\partial\Gamex)}+\|\Sen-S_0\|_{1,\alp,\Gamen}+\delta\|(\til{\phi},  \til{\Psi})\|_{2,\alp,\Om_L}^{(-1-\alp,\corners)}\right).
\end{split}
\end{equation}
By Proposition \ref{proposition-2} and \eqref{4-a7}, the boundary value problem for the elliptic system
\begin{equation}
\label{4-b4}
\left\{
\begin{aligned}
&\mclL_1(\phi,\Psi)=\til f+\Div \til{\bf F}\\
&\mclL_2(\phi,\Psi)=\til g
\end{aligned}
\right.
\quad \tx{in}\quad\Om_L
\end{equation}
with boundary conditions
\begin{equation}\label{4-b4-bc1}
\begin{aligned}
&\Psi_{x_1}=g_{en}\,\,\tx{on}\quad \Gamen,\quad \Psi=\Psi_{bd}\quad \tx{on}\quad \Gamex,\quad \Psi_{x_2}=0\quad \tx{on}\quad \Gamw
\end{aligned}
\end{equation}
and
\begin{equation}\label{4-b4-bc2}
\begin{aligned}
&\phi_{x_1}=0\quad \tx{on}\quad \Gamen,\quad  \phi=\int_{-1}^{x_2}\til h(z)\;dz\quad \tx{on}\quad \Gamex,\\
&\phi(x_1,-1)=0,\quad \phi(x_1,1)=\int_{-1}^1 \til h(z)\;dz\\
\end{aligned}
\end{equation}
has a unique solution $(\phi,\Psi)\in [C^{1,\alp}(\ol{\Om_L})\cap C^{2}(\Om_L)]^2$. Furthermore, we have
\begin{equation}
\label{4-b1}
\|(\phi, \Psi)\|_{2,\alp,\Om_L}^{(-1-\alp,\corners)}
\le \mathscr{C}^*\left(\frac{\om_1(G_{en}, \Phi_{bd},\Sen,\mathscr{B}_{en}, \pex)}{\delta}+ \|(\til{\phi},  \til{\Psi})\|_{2,\alp,\Om_L}^{(-1-\alp,\corners)}\right)\delta,
\end{equation}
where the constant $\mathscr{C}^*>0$ depends on the background data, $\delta_0$ and $\alp$.

Suppose that $\om_1(G_{en}, \Phi_{bd},\Sen,\mathscr{B}_{en}, \pex)\le \sigma$ with $\sigma>0$ to be determined. Set
$$
\delta=M\sigma
$$
for $M\ge 1$ to be determined. Then, \eqref{4-b1} implies
\begin{equation*}
\|(\phi, \Psi)\|_{2,\alp,\Om_L}^{(-1-\alp,\corners)}
\le \mathscr{C}^*(\frac 1M+M\sigma)\delta.
\end{equation*}
Choose $\sigma_1<1$ and $M>1$ as follows:
\begin{equation}
\label{choice}
M=4(\mathscr{C}^*+1),\quad \sigma_1= \min\{\frac{1}{M}, \frac{\delta_1}{M}, \frac{\delta_2}{M}, \frac{1}{M^2}, \frac{1}{2\mathscr{C^*}M}\}.
\end{equation}
Under such choices of $M$ and $\sigma_1$, if $\om_1(G_{en}, \Phi_{bd},\Sen,\mathscr{B}_{en}, \pex)\le \sigma\le \sigma_1$,
then we have
\begin{equation*}
\|(\phi, \Psi)\|_{2,\alp,\Om_L}^{(-1-\alp,\corners)}
\le \frac{\delta}{2}.
\end{equation*}
Thus we can define a mapping $\mcl{J}:\mcl{K}_{\delta}\rightarrow \mcl{K}_{\delta}$ by
\begin{equation*}
\mcl{J}(\til{\phi}, \til{\Psi})=(\phi, \Psi)
\end{equation*}
for the solution $(\phi, \Psi)$ to \eqref{4-b4}--\eqref{4-b4-bc2}.
Once we show that $\mcl{J}$ has a fixed point in $\mcl{K}_{\delta}$, then Theorem \ref{theorem-1} (a) is proved.

Suppose that a sequence $\{(\phi_k,\Psi_k)\}_{k=1}^{\infty}\subset \mcl{K}_\delta$ converges to $(\phi_{\infty},\Psi_{\infty})$ in $[C^{1,\frac{\alp}{2}}(\ol{\Om_L})]^2$, then we have $(\phi_{\infty},\Psi_{\infty})\in\mcl{K}_\delta$. For each $k\in\mathbb{N}$, set
\begin{equation}
\label{4-b6}
\begin{split}
&\mathscr{L}_k:=\mathscr{L}^{(\psi_0+\phi_k)},\;\; \mathscr{L}_{\infty}:=\mathscr{L}^{(\psi_0+\phi_{\infty})},\;\; S_k:=\Sen\circ\mathscr{L}_k,\;\; S_{\infty}:=\Sen\circ\mathscr{L}_{\infty}
\end{split}
\end{equation}
where $\mathscr{L}^{(\psi)}$ is defined by \eqref{2-a5}. By \eqref{2-a4} and \eqref{2-a5}, $\mathscr{L}_k-\mathscr{L}_{\infty}$ can be expressed as
\begin{equation}
\label{4-c2}
(\mathscr{L}_k-\mathscr{L}_{\infty})(\rx)=
\frac{(\phi_k-\phi_{\infty})(\rx)-(\phi_k-\phi_{\infty})(0,\mathscr{L}_k(\rx))}
{\int_0^1 \der_{x_2}(\psi_0+\phi_{\infty})(0,\mathscr{L}_{\infty}(\rx)
+t(\mathscr{L}_k-\mathscr{L}_{\infty})(\rx))\;dt}.
\end{equation}
Thus
$
\underset{k\to\infty}{\lim}\|\mathscr{L}_k-\mathscr{L}_{\infty}\|_{0,\Om_L}=0.
$
Furthermore, using \eqref{2-a7} gives
\begin{equation}
\label{4-b7}
\lim_{k\to\infty}\|\mathscr{L}_k-\mathscr{L}_{\infty}\|_{1,\Om_L}\le
C\lim_{k\to\infty}
\left(\left(\|\mathscr{L}_k-\mathscr{L}_{\infty}\|_{0,\Om_L}\right)^{\alp/2}
+\|\phi_k-\phi_{\infty}\|_{1,\alp/2,\Om_L}\right)=0.
\end{equation}
This, together with \eqref{4-b6}, yields
\begin{equation}
\label{4-b8}
\lim_{k\to\infty} \|S_k-S_{\infty}\|_{\alp/2,\Om_L}=\lim_{k\to\infty} \|\der_{x_2} S_k-\der_{x_2} S_{\infty}\|_{0,\Om_L}=0
\end{equation}
It follows from \eqref{2-b10}--\eqref{2-b11}, \eqref{2-c7}, \eqref{4-a6}, \eqref{4-b4}, \eqref{4-b8} and Lemma \ref{lemma-5} that we have
\begin{equation*}
\lim_{k\to\infty}\|\mcl{J}(\phi_k,\Psi_k)
-\mcl{J}(\phi_{\infty},\Psi_{\infty})\|_{1,\alp/2,\Om_L}=0.
\end{equation*}
Therefore, the mapping $\mcl{J}:\mcl{K}_\delta\to \mcl{K}_\delta$ is continuous in $[C^{1,\alp/2}(\ol{\Om_L})]^2$. Since $\mcl{K}_\delta$ is a closed, compact and convex subset of $[C^{1,\alp/2}(\ol{\Om_L})]^2$, Schauder fixed point theorem implies that $\mcl{J}$ has a fixed point $(\phi^*,\Psi^*)$ in $\mcl{K}_\delta$. Let $S^*$ be the solution to \eqref{eqS} associated with $\phi=\phi^*$. Then, $(\phi^*, \Psi^*, S^*)$ is a solution to \eqref{2-b5}, \eqref{2-b6}, \eqref{eqS}, \eqref{2-c7}, and \eqref{2-d1bc}. In addition, the estimate \eqref{4-b1} together with the choices of $M$ and $\sigma$ yields the estimate \eqref{1-f12}.
\smallskip

{\textbf{Proof of Theorem \ref{theorem-1} (b):}}
Let $(\phi_1,\Psi_1,S_1)$ and $(\phi_2,\Psi_2,S_2)$ be two solutions to \eqref{2-b5}, \eqref{2-b6}, \eqref{eqS}, \eqref{2-c7}, and \eqref{2-d1bc} which satisfy
\begin{equation}
\label{4-c4}
\|(\phi_j, \Psi_j)\|_{2,\alp,\Om_L}^{(-1-\alp,\Gamw)}+\|S_j-S_0\|_{1,\alp,\Om_L}\le C\om_1(G_{en}, \Phi_{bd},\Sen,\mathscr{B}_{en},\pex).
\end{equation}

By the assumption \eqref{2-a7-2}, we have $\om_2(\Phi_{bd},\Sen,\mathscr{B}_{en})=\|\Sen-S_0\|_{W^{2,\mu}(\Gamen)}$.
Given $\alp\in(\frac 12,1)$ and $\mu>2$, choose $\mu_1 \in (2, \min(\mu, \frac{1}{1-\alpha}))$ and denote
$
\beta=\frac 12\min(\alp,1-\frac {2}{\mu}).
$

For each $j=1,2$, $S_j$ can be represented as
\begin{equation}
\label{4-c3}
S_j(\rx)=\Sen\circ\mathscr{L}^{(\psi_0+\phi_j)}(\rx)\quad\tx{in}\quad \Om_L.
\end{equation}
For $j=1,2$, one has
\begin{align*}
&f_j=f(\rx,\nabla\phi_j,\Psi_j, S_j, \der_{x_2} S_j),\;\;
(g_j, {\bf F}_j)=(g, {\bf F})(\rx,\nabla\phi_j,\Psi_j, S_j),\;\;
h_j=h(\rx,\nabla\phi_j,S_j).
\end{align*}
Then $(\hat\phi,\hat\Psi)=(\phi_1-\phi_2,\Psi_1-\Psi_2)$ satisfy the system
\begin{equation*}
\begin{cases}
\mclL_1(\hat\phi,\hat\Psi)=(f_1-f_2)+\Div ({\bf F}_1-{\bf F}_2)\\
\mclL_2(\hat\phi,\hat\Psi)=g_1-g_2
\end{cases}\quad \tx{in}\quad\Om_L
\end{equation*}
with boundary conditions
\begin{equation*}
\begin{aligned}
&\hat{\Psi}_{x_1}=0\,\,\tx{on}\quad \Gamen,\quad \hat\Psi=0\,\,\tx{on}\quad \Gamex,\quad \hat\Psi_{x_2}=0\,\,\tx{on}\quad \Gamw,\\
&\hat\phi_{x_1}=0\,\, \tx{on}\quad \Gamen,\,\, \hat\phi(\rx)=\int_{-1}^{x_2}(h_1-h_2)(z)\;dz\,\,\tx{on}\quad \Gamw\cup\Gamex.
\end{aligned}
\end{equation*}
Suppose that $\om_1(G_{en}\Phi_{bd},\Sen,\mathscr{B}_{en},\pex)
+\om_2(\Phi_{bd},\Sen,\mathscr{B}_{en})\le \hat{\sigma}$ for $\hat{\sigma}>0$ to be determined later.
It follows from \eqref{4-c2} and \eqref{4-c3} that
\begin{equation}
\label{4-c5}
\|S_1-S_2\|_{\beta,\Om_L}\le C\hat{\sigma}\|\phi_1-\phi_2\|_{1,\beta,\Om_L},
\end{equation}
for $C>0$ depending on the background data, $\alpha$, and $\mu$.
For the rest of the section, a constant $C$ may vary but depends only on the background data, $\alpha$, and $\mu$ unless otherwise specified.
For convenience, denote $\mathscr{L}^{(\psi_0+\phi_j)}$ by $\mathscr{L}_j$. Then $\der_{x_2}(S_1-S_2)$ can be written as
\begin{equation}
\label{4-c6}
\der_{x_2}(S_1-S_2)=
\left(\Sen'(\mathscr{L}_1)-\Sen'(\mathscr{L}_2)\right)\der_{x_2}\mathscr{L}_1
+\Sen'(\mathscr{L}_2)\der_{x_2}(\mathscr{L}_1-\mathscr{L}_2).
\end{equation}
It follows from \eqref{4-c2} that
\begin{equation*}
\begin{aligned}
\|\Sen'(\mathscr{L}_1)-\Sen'(\mathscr{L}_2)\|_{L^\mu(\Om_L)}
\le& C\om_2(\Phi_{bd},\Sen,\mathscr{B}_{en})
\|\mathscr{L}_1-\mathscr{L}_2\|_{L^{\infty}(\Om_L)}\\
\le& C\hat{\sigma}\|\phi_1-\phi_2\|_{0,\Om_L}.
\end{aligned}
\end{equation*}
If $\alp>\frac 12$, then
\begin{equation}
\label{4-c8}
\|\der_{x_2}\phi_2(0,\mathscr{L}_2)-\der_{x_2}\phi_2(0,\mathscr{L}_1)\|_{L^{\mu_1}(\Om_L)}
\le C\|\phi_2\|_{2,\alp,\Om_L}^{(-1-\alp,\corners)}
\|\mathscr{L}_1-\mathscr{L}_2\|_{0,\Om_L}.
\end{equation}
Combining \eqref{4-c6}--\eqref{4-c8} with \eqref{4-c4} gives
\begin{equation*}
\|\der_{x_2}(S_1-S_2)\|_{L^{\mu_1}(\Om_L)}\le {C}\hat{\sigma}\|\phi_1-\phi_2\|_{1,\beta,\Om_L}.
\end{equation*}
This, together with \eqref{4-c5},  gives
\begin{equation}
\label{4-d2}
\begin{split}
&\|(f_1-f_2,g_1-g_2)\|_{L^{\mu_1}(\Om_L)}
+\|{\bf F}_1-{\bf F}_2\|_{\beta,\Om_L}+\|h_1-h_2\|_{\beta,\Gamex}
\le  C\hat{\sigma}
\|(\hat{\phi},\hat{\Psi})\|_{1,\beta,\Om_L}.
\end{split}
\end{equation}
Hence it follows  from Lemma \ref{lemma-5}, Remark \ref{remark-1} and \eqref{4-d2} that
\begin{equation}
\label{4-d4}
\|(\hat{\phi},\hat{\Psi})\|_{1,\beta,\Om_L}
\le C_{\natural}\hat{\sigma}
\|(\hat{\phi},\hat{\Psi})\|_{1,\beta,\Om_L}
\end{equation}
for $C_{\natural}$ depending on the background data, $\alp$, $\mu$ and $\delta_0$.
Choose $\sigma_2=\min(\sigma_1, \frac{3}{4 {C}_{\natural}})$ so that, whenever $\hat{\sigma}\le \sigma_2$, \eqref{4-d4} implies
\begin{equation*}
(\phi_1,\Psi_1)=(\phi_2,\Psi_2)\quad\tx{in}\quad \Om_L.
\end{equation*}
This finishes the proof for Theorem \ref{theorem-1}.
\hfill$\Box$

\bigskip

{\bf Acknowledgement.}
The research of Myoungjean Bae and Ben Duan was supported in part by Priority Research Centers Program through the National Research Foundation of Korea(NRF) (NRF Project no.2013053914). The research of Myoungjean Bae was also supported by the Basic Science Research Program(NRF-2012R1A1A1001919) and TJ Park Science Fellowship of POSCO TJ Park Foundation.
The research of Xie was supported in part by  NSFC  11201297, Shanghai Chenguang program and Shanghai Pujiang  program 12PJ1405200, and the Program for Professor of Special Appointment (Eastern Scholar) at Shanghai Institutions of Higher Learning.
Part of the work was done when Xie was visiting POSTECH and Bae and Duan visited  Shanghai Jiao Tong University. They thank these institutions for their hospitality and support during these visit.

\end{document}